\documentclass[11pt,leqno,oneside]{amsart}
\usepackage{amsmath,amsfonts,amssymb,amsthm,enumerate,mathtools}
\usepackage{color,framed,mathscinet}
\usepackage[colorlinks, linkcolor=teal, citecolor=violet]{hyperref}
\usepackage[a4paper, left=2cm, right=2cm, top=3cm, bottom=3cm]{geometry}
\usepackage[numbers,sort&compress]{natbib}
\usepackage{doi}
\usepackage[table]{xcolor}
\usepackage[a-2u]{pdfx}
\usepackage{comment}
\usepackage{multirow}
\usepackage{enumitem}
\setlist[enumerate]{label={\rm(\roman*)},leftmargin=6ex}
\newcommand{\R}{\mathbb{R}}
\newcommand{\rn}{{\mathbb{R}^n}}
\newcommand{\N}{\mathbb{N}}
\newcommand{\MM}{\mathcal{M}}
\newcommand{\FF}{\mathcal{F}}
\newcommand{\EE}{\mathcal{E}}
\newcommand{\RR}{\mathcal{R}}
\newcommand{\HH}{\mathcal{H}}
\newcommand{\DD}{\mathcal{D}}
\newcommand{\Hg}{\HH^{n-1}_{\gamma_n}}
\renewcommand{\d}{{\mathrm d}}
\newcommand{\dHg}{\d\Hg}

\newcommand{\kappab}{\kappa_\beta}
\newcommand{\med}{\operatorname{med}}
\newcommand{\mv}{\operatorname{mv}}
\newcommand{\m}{\operatorname{m}}
\newcommand{\sgn}{\operatorname{sgn}}

\DeclareMathOperator{\Exp}{\operatorname{Exp}}

\let\tilde\widetilde

\usepackage{xspace}
\makeatletter
\DeclareRobustCommand\onedot{\futurelet\@let@token\@onedot}
\def\@onedot{\ifx\@let@token.\else.\null\fi\xspace}
\def\eg{e.g\onedot} 
\def\ie{i.e\onedot} 
 
\def\ae{a.e\onedot} 
\makeatother

\newtheoremstyle{MyPlain}{}{}{\itshape}{}{\bfseries}{.}{5pt plus 4pt minus 3pt}{\thmname{#1}\thmnumber{ #2}\thmnote{ \textbf{[#3]}}}
\theoremstyle{MyPlain}
\newtheorem{theorem}{Theorem}[section]
\newtheorem{theoremalph}{Theorem}

\newtheorem{lemma}[theorem]{Lemma}
\newtheorem{proposition}[theorem]{Proposition}

\newtheoremstyle{MyRemark}{}{}{\upshape}{}{\bfseries}{.}{5pt plus 1pt minus 1pt}{}
\theoremstyle{MyRemark}
\newtheorem{remark}[theorem]{Remark}

\numberwithin{equation}{section}

\expandafter\let\expandafter\oldproof\csname\string\proof\endcsname
\let\oldendproof\endproof
\renewenvironment{proof}[1][\proofname]{%
  \oldproof[{{\bf #1.}}]%
}{\oldendproof}

\makeatletter
\newcommand{\opnorm}{\@ifstar\@opnorms\@opnorm}
\newcommand{\@opnorms}[1]{%
	\left|\mkern-1.5mu\left|\mkern-1.5mu\left|
	#1
	\right|\mkern-1.5mu\right|\mkern-1.5mu\right|
}
\newcommand{\@opnorm}[2][]{%
  \mathopen{#1|\mkern-1.5mu#1|\mkern-1.5mu#1|}
  #2
  \mathclose{#1|\mkern-1.5mu#1|\mkern-1.5mu#1|}
}
\makeatother

\newcommand{\ib}{\frac{1}{\beta}}
\newcommand{\iib}{\frac{2}{\beta}}

\newcommand{\RG}{(\rn,\gamma_n)}
\newcommand{\WexpLb}{W^1\exp L^\beta\RG}
\newcommand{\expLb}{\exp L^\beta\RG}
\newcommand{\dgn}{\d\gamma_n}

\newcommand{\normI}[1]{\opnorm*{\frac{1}{\,I\,}}_{{#1}\bigl(\Phi(t),\frac12\bigr)}}
\newcommand{\normIB}{\normI{L^{\tilde B}}}

\newcommand{\EF}{\varphi}

\newcommand{\expb}{\Exp^\beta}

\makeatletter
\def\paragraph{\bigskip\@startsection{paragraph}{4}%
  \z@\z@{-\fontdimen2\font}%
  {\normalfont\bfseries}}
\makeatother

\begin{document}

\title{On the existence of extremals for Moser type inequalities in Gauss space}

\begin{abstract}
The existence of  an extremal in an exponential Sobolev type inequality, with
optimal constant, in Gauss space is established. A key step in the proof is an
augmented version of the relevant inequality, which, by contrast, fails for a
parallel classical inequality by Moser in the Euclidean space.
\end{abstract}

\author{Andrea Cianchi\textsuperscript{1}}
\address{\textsuperscript{1}Dipartimento di Matematica e Informatica ``Ulisse Dini'',
University of Florence,
Viale Morgagni 67/A, 50134
Firenze,
Italy}
\email{andrea.cianchi@unifi.it}
\urladdr{0000-0002-1198-8718}

\author{V\'\i t Musil\textsuperscript{1}}
\email{vit.musil@unifi.it}
\urladdr{0000-0001-6083-227X}

\author{Lubo\v s Pick\textsuperscript{2}}
\address{\textsuperscript{2}Department of Mathematical Analysis,
Faculty of Mathematics and Physics,
Charles University,
So\-ko\-lo\-vsk\'a~83,
186~75 Praha~8,
Czech Republic}
\email{pick@karlin.mff.cuni.cz}
\urladdr{0000-0002-3584-1454}


\subjclass[2020]{46E35, 28C20}
\keywords{
Gaussian Sobolev inequalities; Gauss measure;
Exponential inequalities; Existence of extremals;
Moser type inequalities;
Orlicz spaces.}

\maketitle

\bibliographystyle{abbrv_doi}

\section*{How to cite this paper}
\noindent
This paper has been accepted to \emph{International Mathematics Research Notices}
and the final publication is available at
\begin{center}
	\url{https://doi.org/10.1093/imrn/rnaa165}.
\end{center}
Should you wish to cite this paper, the authors would like to cordially ask you
to cite it appropriately.

\section{Introduction and main results}

A celebrated result by Moser \cite{Mos:70}, dealing with the borderline case of
the Sobolev embedding theorem \cite{Poh:65, Tru:67, Yud:61}, asserts that if
$\Omega$ is an open set in $\rn$ with finite Lebesgue measure $|\Omega|$, then
\begin{equation} \label{supmoser}
	\sup_u \, \int_{\Omega} \exp^{n'}(\alpha_n |u|)\, \d x < \infty\,,
\end{equation}
where the supremum is extended over all functions $u \in W^{1,n}_0(\Omega)$
satisfying the constraint
\begin{equation} \label{supmoser1}
	\int_{\Omega} |\nabla u|^n  \,\d x \le 1\,.
\end{equation}
Here, $\exp^\beta (t) = e^{t^\beta}$ for $\beta>0$ and $t\geq 0$, and
$\alpha_n = n\omega_n^{1/n}$, where $\omega_n$ denotes the Lebesgue measure of
the unit ball in $\rn$. Moreover, the constant $\alpha_n$ in equation
\eqref{supmoser} is sharp, in the sense that the supremum  is infinite if
$\alpha_n$ is replaced by any larger constant.

An additional remarkable feature of inequality \eqref{supmoser} is that, for
any set $\Omega$ as above, the supremum is, in fact, a maximum. Namely, there
exists a function $u$ at which the supremum is attained. The first contribution
in this connection is \citep{Car:86}, where the case when $\Omega$ is a ball is
considered. The result for arbitrary domains is established in \citep{Flu:92}
for $n=2$, and in \citep{Lin:96} for any $n \geq 2$. However, the supremum in
\eqref{supmoser} and its extremals are still unknown, even in   a ball. A radially symmetric extremal is shown to exist  in this special domain,
but  the existence of additional non-symmetric extremals is not excluded.
Notice that, by contrast, extremals are absent in the classical Sobolev inequality
for $W^{1,p}_0(\Omega)$, for $1<p<n$, with optimal exponent and constant, in
any domain $\Omega\neq\rn$ -- see, for example, \citep[Chapter~I, Section~4.7]{Str:90}.

A major difficulty in the proof of the existence of an extremal in
\eqref{supmoser} is related to the lack of compactness of the embedding
\begin{equation}\label{trudinger}
	W^{1,n}_0(\Omega) \to \exp L^{n'}(\Omega).
\end{equation}
Here, given $\beta>0$, we denote by $\exp L^{\beta}(\Omega)$ the Orlicz space
associated with a Young function equivalent to $\exp^\beta (t)$  near infinity.
To be more specific, sequences $\{u_k\} \subset W^{1,n}_0(\Omega)$ such that
$\int_{\Omega} |\nabla u_k|^n  \,\d x \le 1$ need not  enjoy the property that
the sequence $\{\exp^{n'}(\alpha_n |u_k|)\}$ be uniformly integrable.
This prevents one from applying the classical direct methods of the calculus of
variations to pass to the limit in $\int_{\Omega} \exp^{n'}(\alpha_n
|u_k|)\, \d x$ as $k \to \infty$. The proof of the existence of extremals thus calls
for the use of concentration-compactness techniques.

Moser's inequality has inspired a number of investigations on sharp
exponential inequalities associated with borderline embeddings of Sobolev
type.  The contributions  \citep{Ada:88, Alb:08, Alv:96, Bal:03, Bec:93,
Bra:13, Cer:10, Cia:05, Cia:08, Coh:01, Fon:93, Fon:11, Fon:12, Fon:18,
Hen:03, Ish:11, Kar:16, Lam:13, Lec:05, Li:05, Li:08, Mas:15, Par:19, Ruf:05,
Yan:12} just supply  a taste of this rich line of research.

Unconventional counterparts of Moser's inequality in Gauss space have recently
been offered in \citep{Cia:19}. Recall that the Gauss space $\RG$ is $\rn$
endowed with the Gauss probability measure $\gamma_n$ obeying
\begin{equation*}
	\dgn(x) = (2\pi)^{-\frac{n}{2}} e^{-\frac{|x|^2}{2}}\,\d x
		\quad\text{for $x\in \rn$}.
\end{equation*}
One version of the Gaussian exponential inequalities in question tells us that,
given any $\beta \in (0, 2]$ and $M>1$,
\begin{equation}\label{E:sup-integral}
	\sup_u\, \int_{\rn} \exp^{\frac{2\beta}{2+\beta}}(\kappab |u|)\,\dgn < \infty,
\end{equation}
where the supremum is extended over all weakly differentiable functions
$u\colon\rn\to\R$ fulfilling the gradient bound
\begin{equation} \label{E:nabla-integral}
	\int_{\rn} {\rm Exp}^\beta(|\nabla u|)\,\dgn
		\le M
\end{equation}
and the normalization condition
\begin{equation}\label{E:mu}
	\m(u)=0,
\end{equation}
see \cite[Theorem 1.1]{Cia:19}.  Here,
\begin{equation} \label{E:kappab}
	\kappab = \frac{1}{\sqrt{2}}+\frac{\sqrt{2}}{\beta},
\end{equation}
$\m(u)$ stands for either the mean value $\mv (u)$, or the
median $\med (u)$ of $u$ over $\RG$ (see Section~\ref{sec:symm}
for a precise definition), and
the function $\Exp^\beta$ is the
convex envelope of $\exp^\beta$, which obviously agrees with
$\exp^\beta$   near infinity for every $\beta>0$, and globally if $\beta \geq 1$.

The constant $\kappab$ is sharp in \eqref{E:kappab}, in an even stronger sense
than $\alpha_n$ in \eqref{supmoser}. Indeed, if $\kappab$ is replaced by any
larger constant, then for every $M>1$ there exists a function $u$, fulfilling
conditions \eqref{E:nabla-integral} and \eqref{E:mu}, such that
\begin{equation*}
	\int_{\rn} \exp^{\frac{2\beta}{2+\beta}}(\kappab |u|)\,\dgn
		= \infty.
\end{equation*}
Another diversity between the Euclidean and the Gaussian inequalities is in
that the value of $M$ in \eqref{E:nabla-integral} is surprisingly irrelevant,
whereas the value $1$ in \eqref{supmoser1} is critical.

Inequality \eqref{E:sup-integral} provides us with quantitative information on
the Gaussian Sobolev embedding
\begin{equation}\label{expemb}
	W^1\exp L^\beta\RG \to \exp L^{\frac{2\beta}{2+\beta}}\RG\,,
\end{equation}
where $W^1\exp L^\beta\RG$ denotes the Sobolev space built upon the Orlicz
space $\exp L^\beta\RG$. Embedding \eqref{expemb} is a parallel of
\eqref{trudinger} in Gauss space. Notice that the spaces $\exp L^\beta\RG$
and $\exp L^{\frac{2\beta}{2+\beta}}\RG$ appearing in \eqref{expemb} are
optimal \cite[Proposition~4.4 (iii)]{Cia:09}. In particular, in contrast with
\eqref{trudinger}, the degree of integrability of a function $u$  can be
weaker than that of its gradient $|\nabla u|$ in Gauss space. This is due to
the decay of the measure $\gamma_n$ near infinity.

Let us point out that  embedding \eqref{expemb} does hold, in fact, for every
$\beta >0$, as shown in \citep{Cia:09}. Yet, if $\beta >2$, the validity of
inequality  \eqref{E:sup-integral} turns out to be sensitive to the value of
$M$: the supremum is still finite provided that $M$ is small enough, whereas it
becomes infinite if $M$ is too large. This phenomenon is pointed out in
\citep[Theorem 1.1]{Cia:19}.  The exact threshold for $M$ can only be
characterized in a very implicit form, which does not allow for its
identification. It is therefore unclear for which constants $M$ problem
\eqref{E:sup-integral}--\eqref{E:mu} is meaningful.

Embeddings \eqref{expemb} continue, at a scale of exponential
integrability, the family of classical Gaussian embeddings
\begin{equation} 
	W^{1,p}\RG \to L^p(\log L)^{\frac p2}\RG,
\end{equation}
for $p\in [1, \infty)$, where $ L^p(\log L)^{\frac p2}\RG$ denotes the Orlicz
space built upon  any Young function equivalent to $t^p (\log t)^{\frac p2}$
near infinity. They have their roots in the seminal paper by Gross
\citep{Gro:75}, dealing with the case $p=2$. The extension to $p \neq 2$ can be
found in  \citep{Ada:79}. Further developments and related results are the
subject of a vast literature, including \citep{Ada:79, Bar:06, Bar:08, Bob:98, Bra:07,
Bob:97, Car:01, Cip:00, Fei:75, Fuj:11, Mil:09, Pel:93, Rot:85}.

The purpose of the present paper is to investigate  the existence of an
extremal in the Gaussian exponential inequality \eqref{E:sup-integral}. We give
an affirmative answer to this question, thus providing an analogue of the
existence result for Moser's inequality in the Euclidean space. Moreover, any
possible maximizer $u$ of \eqref{E:sup-integral} is shown to be necessarily a
one-variable function. By contrast, as mentioned above, any (yet qualitative)
characterization of extremals seems to be missing for Moser's inequality.

\begin{theorem}[Existence of maximizers] \label{T:maximizers-integral}
Let $\beta\in(0,2]$ and $M>1$. Then the supremum in  \eqref{E:sup-integral} is
attained.  Moreover, the level sets of  any extremal function $u$ are
half-spaces; namely, there exist an increasing function  $h\colon\R\to\R$ and
$\xi\in\rn$ such that
\begin{equation} \label{sep220}
	u(x) = h (x \cdot \xi)
		\quad\text{for \ae~$x\in \rn$.}
\end{equation}
Here, the dot $``\cdot"$ stands for scalar product in $\rn$.
\end{theorem}

A key novelty in the proof of Theorem~\ref{T:maximizers-integral} is that,
unlike the Euclidean case, the   maximization problem  for the Gaussian
inequality \eqref{E:sup-integral} can be attacked via the direct methods of the
calculus of variations. This is possible thanks to an improvement, of
independent interest, of inequality \eqref{E:sup-integral}. Despite the fact
that, like \eqref{trudinger}, embedding \eqref{expemb} is non-compact, we are
able to show that a uniform integrability property, somewhat stronger than
\eqref{E:sup-integral}, holds under the same constraints
\eqref{E:nabla-integral} and \eqref{E:mu}. Specifically, we prove that if
$\EF\colon[0, \infty)\to [0,\infty)$  is  an increasing function that diverges
to $\infty$ as $t\to\infty$ with a sufficiently mild growth, then
\begin{equation}\label{E:sup-improved}
	\sup _u \int_{\rn} \exp^{\frac{2\beta}{2+\beta}}(\kappab |u|)\varphi (|u|)\,\dgn
		< \infty,
\end{equation}
where the supremum is extended over all weakly differentiable functions
$u\colon\rn\to\R$ satisfying conditions \eqref{E:nabla-integral} and
\eqref{E:mu}.
This is the content of our second main result, where an essentially optimal
growth condition on $\varphi$ for \eqref{E:sup-improved} to hold is
exhibited.

\begin{theorem}[Improved integrability] \label{T:integral-form-improved}
Let $\EF\colon[0,\infty)\to [0,\infty)$
be a continuous increasing function.
Assume that either $\beta\in(0,2)$ and
\begin{equation} \label{E:improvement-conditions-beta-less-than-2}
		\int^\infty t^{-\frac{4}{4-\beta^2}}\EF(t) \,\d t <\infty,
\end{equation}
or $\beta=2$ and
\begin{equation} \label{E:improvement-conditions-beta-2}
		\int^\infty e^{-\frac{1}{8}(\log t)^2}\EF(t) \,\d t <\infty
\end{equation}
Then, inequality \eqref{E:sup-improved} holds for any choice of $M>1$ in
condition \eqref{E:nabla-integral}.
\\
Conversely, if $\beta\in(0,2)$ and
\begin{equation} \label{E:improvement-sharpness-beta-less-than-2}
	\int^\infty t^{-\frac{4+\varepsilon}{4-\beta^2}}
		\EF(t) \,\d t = \infty,
\end{equation}
or $\beta=2$ and
\begin{equation} \label{E:improvement-sharpness-beta-2}
	\int^\infty
		e^{-\frac{1+\varepsilon}{8}(\log t)^2}
		\EF(t) \,\d t = \infty
\end{equation}
for some $\varepsilon>0$, then  inequality \eqref{E:sup-improved} fails for any
choice of $M>1$  in \eqref{E:nabla-integral}, since there exists a
function $u$ satisfying conditions \eqref{E:nabla-integral} and \eqref{E:mu},
and such that
\begin{equation} \label{E:integral-improved-sharp}
	\int_{\rn}
		\exp{(\kappab |u|)^{\frac{2\beta}{2+\beta}}}
			\EF(|u|)\,\dgn
		= \infty.
\end{equation}
\end{theorem}

\begin{remark}
The conclusion of Theorem~\ref{T:integral-form-improved} does not contradict
the fact that $\exp L^{\frac{2\beta}{2+\beta}}\RG$  is the optimal Orlicz
target for embeddings of $W^{1}\exp L^\beta\RG$. Indeed, any  Young function
equivalent to $\Exp^{\frac{2\beta}{2+\beta}}(\kappa t)\EF(t)$ near infinity,
where $\kappa$ is any positive constant and  $\EF$ fulfills condition
\eqref{E:improvement-conditions-beta-less-than-2} or
\eqref{E:improvement-conditions-beta-2}, is equivalent (in the sense of Young
functions) to $\Exp^{\frac{2\beta}{2+\beta}}(t)$ itself near infinity. Hence,
the Orlicz space associated with any Young function of this kind agrees with
$\exp L^{\frac{2\beta}{2+\beta}}\RG$, up to equivalent norms.
\end{remark}

\begin{remark}
Inequality \eqref{E:sup-integral} still holds if the integral condition
\eqref{E:nabla-integral} is replaced by the  Luxemburg norm constraint
\begin{equation} \label{E:nabla-Orlicz}
	\|\nabla u\|_{L^B\RG}\le 1,
\end{equation}
where $B$ is any Young function such that
\begin{equation*}
	N_1e^{t^\beta} \le B(t) \le N_2e^{t^\beta}
		\quad\text{for $t > t_0$,}
\end{equation*}
for some $N_2>N_1>0$ and $t_0>0$ \cite[Theorem~3.1 and Remark~3.5]{Cia:19}.
The existence of maximizers in  \eqref{E:sup-integral} under conditions
\eqref{E:nabla-Orlicz} and \eqref{E:mu} can be established  via an argument
completely analogous to that of the proof of Theorem~\ref{T:maximizers-integral}.
In particular, such an argument relies upon an
improved  version of the result of  \cite[Theorem 3.1]{Cia:19} in the spirit of
Theorem~\ref{T:integral-form-improved}, which tells us that inequality
\eqref{E:sup-improved} holds, when the supremum is extended over all functions
$u$ fulfilling conditions  \eqref{E:nabla-Orlicz} and \eqref{E:mu}, provided
that $\varphi$ is any function as in Theorem~\ref{T:integral-form-improved}.
\end{remark}

\section{Ehrhard symmetrization and ensuing inequalities}\label{sec:symm}

The point of departure of our approach are some rearrangement inequalities for
the gradient of Sobolev functions on Gauss space. These inequalities in their
turn rely upon the isoperimetric inequality that links the Gauss measure of a
set $E\subset \rn$ to its Gauss perimeter. Recall that the Gauss perimeter
$P_{\gamma_n}(E)$ of $E$ can be defined as
\begin{equation*}
	P_{\gamma_n}(E) =
		\Hg(\partial^M E),
\end{equation*}
where we have set
\begin{equation*}
	\dHg(x) = \frac 1{(2\pi)^{\frac n2}}
		 e^{-\frac{|x|^2}2}\,\d\HH^{n-1}(x),
\end{equation*}
$\partial^M E$ denotes the essential boundary of $E$ and $\HH^{n-1}$ is the
$(n-1)$-dimensional Hausdorff measure. The Gaussian isoperimetric inequality
asserts that half-spaces minimize Gauss perimeter among all measurable subsets
of $\rn$ with prescribed Gauss measure \citep{Bor:75, Sud:74}.  Also, as shown
in \citep{Car:01}, half-spaces are the only minimizers.  Note that
\begin{equation}\label{dec10}
	 \gamma_n(\{x\in\rn: x_1\ge t\})= \Phi(t)
		\quad\text{for $t\in\R$,}
\end{equation}
where $x= (x_1, \dots, x_n)$ and $\Phi\colon\R\to(0,1)$ is the function defined as
\begin{equation} \label{E:Phi-def}
	\Phi(t) = \frac{1}{\sqrt{2\pi}} \int_{t}^{\infty} e^{-\frac{\tau^2}{2}}\,\d\tau
		\quad\text{for $t\in\R$.}
\end{equation}
Moreover,
\begin{equation*}
	P_{\gamma_n} (\{x\in\rn: x_1\ge t\})
		=  \frac{1}{\sqrt{2\pi}} e^{-\frac{t^2}{2}}
		\quad\text{for $t\in\R$.}
\end{equation*}
Therefore, on defining the function $I\colon[0,1]\to [0,\infty)$ as
\begin{equation}\label{I}
	I(s) = \frac{1}{\sqrt{2\pi}} e^{-\frac{\Phi^{-1}(s)^2}{2}}
		\quad\text{ for $s\in(0,1)$,}
\end{equation}
and $I(0)=I(1)=0$,
the Gaussian isoperimetric inequality takes the analytic form
\begin{equation} \label{E:isoperimetric-ineq}
	I(\gamma _n(E)) \le P_{\gamma_n}(E)
\end{equation}
for every measurable set $E\subset\rn$, with equality if and only if $E$
agrees with
a half-space (up to a set of measure zero). The function $I$ is accordingly called the
isoperimetric function (or isoperimetric profile) of Gauss space. Note that $I$
is symmetric about $\tfrac 12$, namely
\begin{equation} \label{E:symmetry}
	I(s) = I(1-s)
		\quad\text{for $s\in[0,1]$.}
\end{equation}
Also,
\begin{equation} \label{june30}
	- \Phi'(t) = I(\Phi(t))
		\quad\text{for $t\in\R$},
\end{equation}
where $``\,'\,"$ denotes differentiation.
An Ehrhard symmetral of a function $u\in\MM\RG$ is a function which is
equimeasurable with $u$ and whose level sets are  half-spaces.
Here, $\MM\RG$ denotes the space of measurable functions on $\rn$
with respect to $\gamma_n$. A parallel notation will be employed for the
space of measurable functions on more general measure spaces.
Recall that two
functions, possibly defined on different measure spaces, are called
equimeasurable if all their level sets have like measures.

The signed
decreasing rearrangement $u^\circ\colon(0,1)\to\R$ of a function $u\in\MM\RG$
is defined as
\begin{equation*}
	u^\circ (s) = \inf \{t\in\R: \gamma_n(\{u>t\})\le s\}
		\quad\text{for $s\in(0,1)$.}
\end{equation*}
The functions $u$ and $u^\circ$ are equimeasurable. Thus, the mean value of $u$ satisfies
\begin{equation*}
	\mv(u)
		= \int_{\rn} u \,\dgn 
		= \int_0^1 u^\circ\,\d s.
\end{equation*}
Moreover, the median of $u$ can be defined as
\begin{equation*}
	\med (u) = u^\circ\left(\tfrac 12\right).
\end{equation*}
Thanks to equation \eqref{dec10}, the function $u^\bullet\colon\rn\to\R$,
defined as
\begin{equation*}
	u^\bullet (x) = u^\circ (\Phi (x_1))
		\quad\text{for $x\in\rn$,}
\end{equation*}
is an Ehrhard symmetral of $u$.
Owing to the equimeasurability of $u$, $u^\circ$, $u^\bullet$, one has that
\begin{equation}
	\int_0^1 A(|u^\circ|)\,\d s
		= \int_{\rn} A(|u^\bullet|)\,\dgn
		= \int_{\rn} A(|u|)\,\dgn
\end{equation}
for every increasing function $A\colon[0,\infty)\to[0,\infty]$.

A P\'olya--Szeg\H{o} principle on the decrease of gradient integrals under
Ehrhard symmetrization holds if the integrand $A\colon[0,\infty)\to [0,\infty)$
is a Young function, namely a convex function vanishing at the origin. This
result is established in \cite{Ehr:84}, and is recalled in the next
proposition, where, in addition, a characterization of the cases of equality is
offered. Such a characterization is needed in view of the piece of information
on extremals given in Theorem~\ref{T:maximizers-integral}. A proof seems not to
be available in the literature, and is provided below.

\begin{proposition} \label{P:PSintegral}
Let $A$ be a  Young function.  Assume that $u$ is a weakly differentiable
function in $\rn$ such that
\begin{equation}\label{sep199}
	\int_{\rn} A(|\nabla u|)\,\dgn < \infty.
\end{equation}
Then the function $u^\circ$ is locally absolutely continuous in $(0,1)$, the
function $u^\bullet$ is weakly differentiable in $\rn$, and
\begin{equation}\label{sep200}
	\int_0^1 A\bigl(-u^\circ{}'(s)I(s)\bigr)\,\d s
		= \int_{\rn} A(|\nabla u^\bullet|)\,\dgn
		\le \int_{\rn} A(|\nabla u|)\,\dgn.
\end{equation}
Moreover, if $A$ is strictly positive in $(0, \infty)$ and finite-valued, and
equality holds in the inequality in \eqref{sep200}, then all level sets of $u$
are half-spaces, namely  there exists $\xi\in\rn$, with $|\xi|=1$, such that
\begin{equation}\label{sep201}
	u(x) = u^\circ (\Phi (x \cdot \xi))
		\quad\text{for \ae $x\in\rn$}.
\end{equation}
\end{proposition}

\begin{proof}
The fact that $u^\circ$ is locally absolutely continuous in $(0,1)$,
and hence the function $u^\bullet$ is weakly differentiable in $\rn$,   goes
back to \citep{Ehr:84}  -- see also \citep[Lemma~3.3]{Cia:09}. Assume now that
$u$ fulfills \eqref{sep199}. The coarea formula tells us that
\begin{equation} \label{sep204}
	\int_{\rn} f\, |\nabla u|\, \dgn
		= \int_{-\infty}^\infty \int_{\{u=t\}} f\,\dHg\,\d t
\end{equation}
for every Borel function $f\colon\rn\to\R$, provided that a representative of
$u$ is suitably chosen. In particular, the choice of $f=\chi_{\{\nabla u=0\}}$
in equation \eqref{sep204} tells us that
\begin{equation} \label{sep205}
	\Hg (\{\nabla u= 0\} \cap \{u=t\}) = 0
		\quad\text{for \ae $t\in\R$}.
\end{equation}
Next, set
\begin{equation*}
	\mu(t) = \gamma_n (\{u>t\})
		\quad\text{for $t\in\R$}.
\end{equation*}
We claim that
\begin{equation} \label{sep206}
	\frac{\Hg(\{u^\bullet=t\})}{|\nabla u^\bullet |_{\upharpoonright\{u^\bullet =t\}}}
		= - \mu'(t)
		\geq \int_{\{u=t\}} \frac{1}{|\nabla u|}\,\dHg
		\quad \text{for \ae $t\in\R$},
\end{equation}
where $|\nabla u^\bullet |_{\upharpoonright\{u^\bullet =t\}}$ denotes the
(constant) value of $|\nabla u^\bullet|$ restricted to the set $\{u^\bullet
=t\}$.  Indeed, the inequality in \eqref{sep206} follows from the fact that
\begin{align*}
	\mu (t)
		& = \int_{\{u>t\}} \left( \chi_{\{\nabla u \neq 0\}}
				+ \chi_{\{\nabla u = 0\}}\right) \dgn
			\\
		&	=  \int_t^\infty \int_{\{u=\tau\}} \frac{\chi_{\{\nabla u \neq 0\}}}{|\nabla u|} \,\dHg\,\d \tau
			+ \gamma_n \bigl(\{u>t\} \cap \{\nabla u = 0\}\bigr)
			\\
		& = \int_t^\infty \int_{\{u=\tau\}} \frac{1}{|\nabla u|} \,\dHg\,\d\tau
			+ \gamma_n \bigl(\{u>t\} \cap \{\nabla u = 0\}\bigr)
		\quad\text{for $t\in\R$},
\end{align*}
where the second equality holds by formula \eqref{sep204} and the third one
owing to equation \eqref{sep205}. In order to verify  the equality in  equation
\eqref{sep206}, let us set
\begin{equation*}
	\DD_{u^\bullet}
		= \{x\in\rn:  \,\text{$u^\bullet$ is approximately differentiable at $x$}\}
\end{equation*}
and
\begin{equation*}
	\DD_{u^\circ}
		= \{s\in (0,1): \,\text{$u^\circ$ is approximately differentiable at $s$}\}.
\end{equation*}
Recall that
\begin{equation}\label{nov1}
	\gamma_n(\rn\setminus \DD_{u^\bullet})=0
		\quad\text{and}\quad
	|(0,1) \setminus \DD_{u^\circ}|=0,
\end{equation}
since $u^\bullet$ and $u^\circ$ are weakly differentiable functions. Here,
vertical bars $|\, \cdot \,|$ stand for the Lebesgue measure. In the remaining part
of this proof, let $ \nabla u^\bullet $ and $ {u^\circ}' $  denote the
approximate differentials of $u^\bullet$ and $u^\circ$,   respectively.  Set
\begin{equation*}
	\DD^+_{u^\bullet}
		= \{x\in\rn: \nabla u^\bullet(x)\,\text{exists and}\,\,\nabla u^\bullet(x)\ne 0\},
		\quad
	\DD^0_{u^\bullet}
		=\{x\in\rn: \nabla u^\bullet(x)\,\text{exists and}\,\,\nabla u^\bullet (x) = 0\},
\end{equation*}
and
\begin{equation*}
	\DD^+_{u^\circ}
		=\{s \in (0,1): {u^\circ}' (s) \, \text{exists and}\,\,{u^\circ}'(s)\ne 0\},
	\quad
	\DD^0_{u^\circ}
		=\{s \in (0,1): {u^\circ}' (s) \, \text{exists and}\,\,{u^\circ}'(s) = 0\}.
\end{equation*}
A Gaussian version of \citep[Lemma 3.2]{Cia:02}, with analogous proof, tells
us that the equality in equation \eqref{sep205} holds for \ae $t \in
u^\bullet(D^+_{u^\bullet})$, and that $\mu'(t) = 0$ for \ae $t \notin
u^\bullet(D^+_{u^\bullet})$. Also, one can verify that $u^\bullet$ is
approximately differentiable at $x$ if and only if $u^\circ$ is approximately
differentiable at $x_1$, and that
\begin{equation*}
	\nabla u^\bullet (x)
		= - {u^\circ}' (\Phi(x_1))I(\Phi(x_1))(1, 0, \dots , 0).
\end{equation*}
Hence, one infers that $u^\bullet(D^+_{u^\bullet})=  u^\circ(D^+_{u^\circ})$,
$u^\bullet(D^0_{u^\bullet})=  u^\circ(D^0_{u^\circ})$ and $u^\bullet(\rn
\setminus \DD_{u^\bullet})= u^\circ((0,1) \setminus \mathcal
D_{u^\circ}) $. By \citep[Lemma 2.4]{Cia:02}, $u^\circ(D^0_{u^\circ})=0$.
Also, $u^\circ((0,1) \setminus \DD_{u^\circ}) =0$,
owing to the second equation in \eqref{nov1} and to the fact that absolutely continuous functions map sets of measure zero into sets of measure zero. Altogether, we infer that the equality holds in \eqref{sep206} for \ae $t\in\R$.

Next, the following chain holds:
\begin{align} \label{sep209}
	\begin{split}
	\int_{\{u=t\}}\frac{A(|\nabla u|)}{|\nabla u|}\,\dHg
		& \geq A\left(\frac
				{\int_{\{u=t\}} \frac{|\nabla u|}{|\nabla u|}\, \dHg}
				{\int_{\{u=t\}} \frac{1}{|\nabla u|}\, \dHg}\right)
			\int_{\{u=t\}} \frac{1}{|\nabla u|}\,\dHg
			\\
		& = A\left(\frac
				{\Hg(\{u =t\})  }
				{\int_{\{u=t\}} \frac{1}{|\nabla u|}\, \dHg}\right)
			\int_{\{u=t\}} \frac{1}{|\nabla u|}\,\dHg
			\\
		& \geq A\left(\frac
				{\Hg(\{u =t\})  }
				{\Hg(\{u^\bullet =t\}) }
			|\nabla u^\bullet |_{\upharpoonright\{u^\bullet =t\}}  \right)
			\frac{ \Hg(\{u^\bullet =t\})}
				{|\nabla u^\bullet |_{\upharpoonright\{u^\bullet =t\}}}
			\\
		& \geq A\left( |\nabla u^\bullet |_{\upharpoonright\{u^\bullet =t\}}  \right)
			\frac{\Hg(\{u^\bullet =t\})}
				{|\nabla u^\bullet |_{\upharpoonright\{u^\bullet =t\}}}
			\\
		& = \int_{\{u^\bullet =t\}}
			\frac{A(|\nabla u^\bullet |)}
			{|\nabla u^\bullet |}\,\dHg
		\quad\text{for \ae $t\in\R$.}
		\end{split}
\end{align}
Observe that the first inequality in \eqref{sep209} holds by Jensen's
inequality, the second one by inequality  \eqref{sep206} and the fact that the
function $A(t)/t$ is non-decreasing, and the third one since, owing to the
isoperimetric inequality in Gauss space \eqref{E:isoperimetric-ineq},
\begin{equation}\label{sep210}
	\Hg(\{u =t\})
		=  P_{\gamma _n}(\{u >t\})
		\geq  P_{\gamma _n}(\{u^\bullet >t\})
		=  \Hg(\{u^\bullet =t\})
	\quad\text{for \ae $t\in\R$}.
\end{equation}
An integration of inequality \eqref{sep209} in $t$ over $\R$ and the use of 
coarea formula  \eqref{sep204} yield
\begin{align}\label{sep211}
	\begin{split}
	\int_{\rn} A(|\nabla u|)\, \dgn
		& = \int_{-\infty}^\infty
				\int_{\{u=t\}}\frac{A(|\nabla u|)}{|\nabla u|}\, \dHg \, \d t
		\\
		& \geq \int _{-\infty}^\infty \int_{\{u^\bullet =t\}}
				\frac{A(|\nabla u^\bullet |)}{|\nabla u^\bullet |}\, \dHg \, \d t
		= \int_{\rn} A(|\nabla u^\bullet|)\,\dgn,
	\end{split}
\end{align}
whence equation \eqref{sep200} follows.

As far  as the equality case is concerned, if the inequality in \eqref{sep200}
holds as equality, then the same is true in \eqref{sep211}. If $A$ is
strictly positive in $(0,\infty)$ and finite-valued, then it is strictly
increasing. An inspection of the proof of    inequality   \eqref{sep211}
then reveals that, in particular, equality must hold in inequality
\eqref{sep210}  for \ae $t\in\R$. The characterization of extremal sets in
the  Gaussian isoperimetric inequality implies that  $\{u>t\}$ is a
half-space for \ae $t\in\R$. Since the level sets of any function are
nested, this implies that the level sets of $u$ are half-spaces with parallel
boundaries. Hence, equation \eqref{sep201} follows.
\end{proof}

\section{Proof of Theorem~\ref{T:integral-form-improved}}

Our proof of the enhanced version of inequality \eqref{E:sup-integral},
contained in Theorem~\ref{T:integral-form-improved}, rests upon a precise
estimate for the asymptotic behaviour of a norm in an Orlicz space depending
on the functions $\Phi$ and $I$ introduced in \eqref{E:Phi-def} and \eqref{I}.
We begin by recalling a few facts from the theory of Young functions and Orlicz
spaces that are needed in deriving this estimate.

Let $(\RR, \nu)$ be a non-atomic probability space that, in what follows, will
be either $\rn$ endowed with the Gauss measure $\gamma_n$, or $(0,1)$ endowed
with the Lebesgue measure (in which case  the measure will  be omitted in the
notation).  The Orlicz space $L^A(\RR, \nu)$ built upon a Young function $A$ is
defined as
\begin{equation*} 
	L^A(\RR, \nu)
		= \bigg\{\phi\in \MM(\RR, \nu):
			\int_{\RR} A\biggl(\frac{|\phi |}{\lambda}\biggr)\, \d\nu  < \infty
			\,\,\text{for some $\lambda>0$}\bigg\}.
\end{equation*}
The space $L^A(\RR, \nu)$ is a Banach space equipped with the Luxemburg norm
given by
\begin{equation*}
	\|\phi\|_{L^A(\RR, \nu)}
		= \inf \bigg\{\lambda>0:
			\int_{\RR} A\bigg(\frac{|\phi |}{\lambda}\bigg)\,\d\nu\le 1
		\bigg\}
\end{equation*}
for $\phi \in L^A(\RR, \nu)$.
One has  that $L^A(\RR, \nu)=L^B(\RR, \nu)$ (up to equivalent
norms) if and only if $A$ and $B$ are Young functions equivalent near infinity,
in the sense that $A(c_1t) \le B(t) \le A(c_2t)$  for some positive constants $c_1$
and $c_2$, and for sufficiently large $t$.
Recall that
\begin{equation*} 
	L^\infty (\RR, \nu) \to L^A(\RR, \nu) \to L^1(\RR, \nu)
\end{equation*}
for every Young function $A$.
The Orlicz norm
$\opnorm{\,\cdot \,}_{L^{A}(\RR, \nu)}$, given by
\begin{equation*}
	\opnorm{\phi}_{L^A(\RR, \nu)}
		= \sup\biggl\{ \int_{\RR} \phi\psi\,\d\nu:
				\int_{\RR} \tilde{A}(|\psi|)\,\d\nu\le 1
			\biggr\}
\end{equation*}
for $\phi \in L^A(\RR, \nu)$, is equivalent to the Luxemburg norm. Here,
$\widetilde A\colon[0,\infty)\to[0,\infty]$ denotes the Young conjugate of
$A$, defined as
\begin{equation*} 
	\widetilde A(t)
		= \sup\{\tau t-A(\tau): \tau\ge 0\}
	\quad\text{for $t\ge 0$,}
\end{equation*}
which is also a Young function.  Notice that, if
$a\colon[0,\infty)\to[0,\infty]$ is the  non-decreasing left-continuous function such that
\begin{equation*}
	A(t) = \int_0^t a(\tau)\, \d \tau
	\quad\text{for $t\ge 0$,}
\end{equation*}
then $\widetilde A$ admits the representation formula
\begin{equation*}
	\widetilde A (t) = \int_0^ta^{-1}(\tau)\, \d\tau
		\quad\text{for $t\ge 0$,}
\end{equation*}
where $a^{-1}$ denotes the (generalized) left-continuous inverse of $a$.
By the very definition of Young conjugate, one has that
\begin{equation}\label{dec12}
	t \tau \leq A(t) + \widetilde A(\tau)
		\quad \text{for $t, \tau \geq 0$.}
\end{equation}
Moreover, equality holds in \eqref{dec12} if and only if either $t=
a^{-1}(\tau)$ or $\tau = a(t)$.
If $\phi\in L^{A}(\RR,\nu)$ and  $E$ is a $\nu$-measurable subset of $\RR$, we use the abridged
notation
\begin{equation*}
	\|\phi\|_{L^A(E,\nu)}
		= \|\phi\chi_E\|_{L^A(\RR, \nu)}
	 \quad\text{and}\quad
	 \opnorm{\phi}_{L^A(E, \nu)}
	 	= \opnorm{\phi\chi_E}_{L^A(\RR, \nu)}.
\end{equation*}
In particular, one has that
\begin{equation} 
	\opnorm{1}_{L^A(E, \nu)}
		= \nu(E)\tilde{A}^{-1}\bigl(1/\nu(E)\bigr).
\end{equation}
Here, $\tilde{A}^{-1}$ denotes the (generalized) right-continuous inverse of
$\tilde A$.  A sharp form of the H\"older inequality in Orlicz spaces tells us
that
\begin{equation} 
	\int_{\RR} \phi\psi\,\d\nu
		\le \|\phi\|_{L^A(\RR, \nu)} \opnorm{\psi}_{L^{\tilde A}(\RR, \nu)}
\end{equation}
for every $\phi \in L^A(\RR, \nu)$ and $\psi \in L^{\tilde A}(\RR, \nu)$.

The next lemma provides us with a uniform bound for the integral in
\eqref{E:sup-improved} for functions obeying \eqref{E:nabla-integral}.  Such a
bound involves a function $\FF_B$, which is in its turn associated with a Young
function $B$ obeying
\begin{equation} \label{BN}
	B(t)=Ne^{t^\beta}
		\quad\text{for $t>t_0$}
\end{equation}
for some $N>0$ and $t_0>0$.  The function $\FF_B\colon(0,\infty)\to(0,\infty)$
is defined as
\begin{equation} \label{E:F-LB}
	\FF_B(t)
		= \normIB
		+ \frac{\sqrt{2\pi}}{2}B^{-1}(1)
			\quad\text{for $t>0$}.
\end{equation}

\begin{lemma} \label{L:Holder-estimates}
Let $\beta >0$, $\kappa>0$ and $M>1$. Assume that the function
$\EF\colon[0,\infty)\to[0,\infty)$ is  continuous and such that the function
$\exp^\frac{2\beta}{2+\beta}(\kappa t)\EF(t)$ is non-decreasing. Then
there exists a Young function $B$ of the form \eqref{BN} such that
\begin{equation} \label{E:exponential-estimate}
	\int_{\rn}
		\exp^{\frac{2\beta}{2+\beta}}\left(\kappa |u| \right)
		\EF\left(|u|\right)\,
		\dgn
		\le \sqrt{\frac{2}{\pi}}
			\int_{0}^{\infty} e^{\left[ \kappa\FF_B(t) \right]^{\frac{2\beta}{2+\beta}}
				-\frac{t^2}{2}}
			\EF\left( \FF_B(t) \right)
				\,\d t
\end{equation}
for every weakly differentiable function $u$ in $\rn$ satisfying conditions
\eqref{E:nabla-integral} and \eqref{E:mu}.
\end{lemma}

\begin{proof}
Choose $t_0$ so large that $ \Exp^\beta(t) = \exp ^\beta (t)$ for $t \geq
t_0$, and define the Young function $A\colon[0,\infty)\to[0,\infty)$
as
\begin{equation*}
	A(t) =
	\begin{cases}
		\frac{t}{t_0} \Exp^\beta(t_0)
			& \text{for $t\in [0,t_0)$}
		\\
		\Exp^\beta(t)
			& \text{for $t\in [t_0,\infty)$}.
	\end{cases}
\end{equation*}
 Set  $N=1/(M+\Exp^\beta(t_0))$ and let $B=NA$.  Then $B$ is a Young
function. If $u$ is a function as in the statement, then
\begin{align*}
	\int_{\rn} B(|\nabla u|)\,\dgn
		 \le N \int_{\{|\nabla u|\ge t_0\}} \Exp^\beta(|\nabla u|)\,\dgn
				+ N \int_{\{|\nabla u|<t_0\}} \Exp^\beta(t_0)\,\dgn
		 \le N\bigl(M + \Exp^\beta(t_0)\bigl) = 1.
\end{align*}
Hence, by the very definition of Luxemburg norm, $\|\nabla u\|_{L^B\RG} \le 1$.
Inequality \eqref{E:exponential-estimate} is thus a consequence of
\citep[Lemma~4.5]{Cia:19}.
\end{proof}

The behaviour near infinity of the function $\FF_B$, defined by \eqref{E:F-LB},
is described in Lemma~\ref{L:I-norm-sharp} below. This is the key estimate to
which we alluded above. Its proof requires several  expansions and bounds
related to the functions entering the definition of   $\FF_B$. They are the
content of some preliminary lemmas, whose statements and proofs  need some
notation related to expansions of functions that are fixed hereafter.

Given a function $\FF$ defined in some neighbourhood of
infinity, and $k\in\N$,   we write
\begin{equation}\label{dic30}
	\FF(t) = \EE_1(t) + \cdots + \EE_k(t) + \cdots
		\quad\text{as $t\to\infty$}
\end{equation}
to denote that
\begin{equation*}
	\lim_{t\to\infty} \frac{\FF(t)}
		{\EE_1(t)} = 1
			\quad\text{if  $k=1$},
	\quad\text{and}\quad
	\lim_{t\to\infty} \frac{\FF(t) - [\EE_1(t) + \cdots + \EE_{j}(t)]}
		{\EE_{j+1}(t)} = 1
			\quad\text{for $1\le j\le k-1$ otherwise.}
\end{equation*}
Clearly, if equation \eqref{dic30} holds, then
\begin{equation} \label{E:power}
	\FF(t)^\sigma
		= \EE_1(t)^\sigma + \sigma \EE_1(t)^{\sigma-1} \EE_2(t)
			+ \cdots
		\quad\text{as $t\to\infty$}
\end{equation}
for every $\sigma>0$.
Also,
\begin{equation} \label{E:power2}
	\FF(t)^\sigma
		= \EE_1(t)^\sigma
			+ \sigma \EE_1(t)^{\sigma-1} \EE_2(t)
			+ \sigma \EE_1(t)^{\sigma-1} \EE_2(t)
				\left( \frac{\EE_3(t)}{\EE_2(t)}
					+ \frac{\sigma-1}{2}\frac{\EE_2(t)}{\EE_1(t)} \right)
			+ \cdots
		\quad\text{as $t\to\infty$}.
\end{equation}

The next three lemmas are stated without proofs. They can be obtained by simple
considerations, via  L'H\^opital's rule.

\newcommand{\lowerlimit}{d}

\begin{lemma}\label{L:Psi}
Assume that either $\sigma \in (-1, \infty)$ and $\lowerlimit \geq 1$, or
$\sigma \in (-\infty, -1]$ and $\lowerlimit>1$.  Let $\Psi_\sigma\colon
(\lowerlimit,\infty)\to[0,\infty)$ be the function defined as
\begin{equation} \label{E:Psi}
	\Psi_\sigma(t) = \int_{\lowerlimit}^{t} (\tau^2-1)^\sigma\,\d \tau
		\quad\text{for $t>\lowerlimit$.}
\end{equation}
Then
\begin{equation*}
\renewcommand{\arraystretch}{1.2}
	\Psi_\sigma(t)
	= \left\{\begin{array}{@{}l@{}l@{}ll@{}}
		\mathrlap{\displaystyle c + \cdots}
			&
			&
			& \text{if $\sigma\in(-\infty,-\tfrac12)$}
		\\
		\mathrlap{\displaystyle\log t + c + \cdots}
			&
			&
			& \text{if $\sigma=-\tfrac12$}
		\\
				\multirow{5}{*}{$\left.
				\begin{array}{@{}l@{}}
					\null \\
					\null \\
						\displaystyle
						\frac{1}{2\sigma+1} t^{2\sigma+1}
					\\ \null
					\\ \null
				\end{array}\right\{$}
			& \mathrlap{
				\displaystyle + c + \cdots
				}
			&
			& \text{if $\sigma\in(-\tfrac12,\tfrac12)$}
		\\
			& \mathrlap{
				\displaystyle
				- \frac{1}{2}\log t + c + \cdots
				}
			&
			& \text{if $\sigma=\tfrac12$}
		\\
			&
				\multirow{3}{*}{$\left.
				\begin{array}{@{}l@{}}
					\null \\
					\displaystyle
					- \frac{\sigma}{2\sigma-1} t^{2\sigma-1}
					\\ \null
				\end{array}\right\{$}
			& + c + \cdots
			& \text{if $\sigma\in(\tfrac12,\tfrac32)$}
		\\
			&
			& + \frac{3}{8}\log t + c + \cdots
			& \text{if $\sigma=\tfrac32$}
		\\
			&
			& + \frac{1}{2} \frac{\sigma(\sigma-1)}{2\sigma-3} t^{2\sigma-3}
				+ \cdots
			& \text{if $\sigma\in(\tfrac32,\infty)$}
		\\
    \end{array}\right.
		\quad\text{as $t\to\infty$},
\end{equation*}
where $c$ denotes a constant,  possibly different at different  occurrences,
 depending on $\sigma$ and $\lowerlimit$.
\end{lemma}

\begin{lemma} \label{L:Upsilon}
Assume that either $\sigma \in (-1, \infty)$ and $\lowerlimit\ge 1$, or $\sigma
\in (-\infty, -1]$ and $\lowerlimit>1$.  Let $\Upsilon_\sigma\colon
(a,\infty)\to[0,\infty)$ be the function defined as
\begin{equation}
\label{E:Upsilon}
	\Upsilon_\sigma(t) = \int_{\lowerlimit}^t (\tau^2-1)^\sigma \log(\tau^2-1)\,\d \tau
		\quad\text{for $t>\lowerlimit$.}
\end{equation}
Then
\begin{equation*}
	\Upsilon_\sigma (t)=
		\begin{cases}
			\displaystyle
			c + \cdots
				& \text{if $\sigma \in(-\infty,-\frac12)$}
				\\[\smallskipamount]
			\displaystyle
			(\log t)^2 + \cdots
				& \text{if $\sigma = -\frac12$}
				\\
			\displaystyle
			\frac{2}{2\sigma+1} t^{2\sigma+1} \log t - \frac{2}{(2\sigma+1)^2}t^{2\sigma+1}
				+ \cdots
				& \text{if $\sigma\in(-\frac12,\infty)$}			
		\end{cases}
	\quad\text{as $t\to\infty$},
\end{equation*}
where $c$ is a constant depending on $\sigma$ and $\lowerlimit$.
\end{lemma}

\begin{lemma} \label{L:a-invers}
Let $B$ be a Young function obeying \eqref{BN} for some $\beta >0$ and $N>0$, and let
$b\colon[0,\infty)\to[0,\infty)$ be the left-continuous function such that
\begin{equation}\label{B}
	B(t)=\int_{0}^t b(\tau)\,\d\tau
		\quad\text{ for $t>0$.}
\end{equation}
  Then
\begin{equation*}
	b^{-1}(t)
		= (\log t)^\ib
		+ \frac{1-\beta}{\beta^2}
			(\log t)^{\ib-1}\log\log t
		- \frac{\log N\beta}{\beta} (\log t)^{\ib-1}
		+ \cdots
			\quad\text{as $t\to\infty$}.
\end{equation*}
\end{lemma}

Two crucial steps in view of the proof of Lemma~\ref{L:I-norm-sharp}
are enucleated in Lemma~\ref{L:I-norm-general} and \ref{L:lambda-t-asymp}.

\begin{lemma}
\label{L:I-norm-general}
Let $\Phi$ be the function defined by \eqref{E:Phi-def}.
Let $B$ be any finite-valued Young function of the form \eqref{B}.
Then
\begin{equation} \label{E:I-norm-general}
	\normIB = \int_{0}^{t} b^{-1}\Bigl( \lambda_t e^{\frac{\tau^2}{2}}\Bigr)\,\d \tau
		\quad\text{for $t>0$},
\end{equation}
where $\lambda_t$ is the unique positive number such that
\begin{equation} \label{E:lambda-t-condition}
	\int_{0}^{t} B\left( b^{-1}\Bigl( \lambda_t e^\frac{\tau^2}{2} \Bigr) \right)
		e^{-\frac{\tau^2}{2}}\,\d \tau
		= \sqrt{2\pi}.
\end{equation}
\end{lemma}

\begin{proof}
Fix $t\in(0,\infty)$. By the definition of  Orlicz norm,
\begin{equation}\label{dec31}
	\normIB = \sup\left\{ \int_{\Phi(t)}^{\frac{1}{2}} \frac{g(s)}{I(s)}\,\d s:
		\int_{\Phi(t)}^{\frac{1}{2}} B(|g(s)|)\,\d s \le 1
		\right\}.
\end{equation}
The change of variables $s=\Phi(\tau)$ and  $f=g(\Phi)$ in both integrals in
\eqref{dec31} yield, via \eqref{june30},
\begin{equation*}
	\normIB = \sup\left\{ \int_{0}^{t} f(\tau)\,\d \tau:
		\int_{0}^{t} B\left(|f(\tau)|\right) e^{-\frac{\tau^2}{2}}\,\d \tau\le \sqrt{2\pi}
		\right\}.
\end{equation*}
Given any function $f\in\MM(0,t)$ such that $f\ge 0$ and any
$\lambda_t>0$, from Young's inequality \eqref{dec12} we infer that
\begin{equation} \label{mar8}
	\int_{0}^{t} f(\tau)\,\d \tau
		= \int_{0}^{t} \frac{f(\tau)}{\lambda_t} \lambda_t e^{\frac{\tau^2}{2}}
				e^{-\frac{\tau^2}{2}}\,\d \tau
		\le \int_{0}^{t} B\left( \frac{f(\tau)}{\lambda_t} \right)
				e^{-\frac{\tau^2}{2}}\,\d \tau
			+ \int_{0}^{t} \tilde B\Bigl( \lambda_t e^{\frac{\tau^2}{2}} \Bigr)
				e^{-\frac{\tau^2}{2}}\,\d \tau.
\end{equation}
Define the function $f_t\colon[0,t]\to[0,\infty)$ as
\begin{equation*}
	f_t(\tau) = \lambda_t\, b^{-1} \Bigl(\lambda_t e^{\frac{\tau^2}{2}}\Bigr)
		\quad\text{for $\tau\in[0,t]$}.
\end{equation*}
By the case of equality in Young's inequality \eqref{dec12},
\begin{equation} \label{mar9}
	f_t\left(\tau\right)e^{\frac{\tau^2}{2}}
		= \lambda_t e^{\frac{\tau^2}{2}} b^{-1} \Bigl(\lambda_t e^{\frac{\tau^2}{2}}\Bigr)
		= B\left( \frac{f_t(\tau)}{\lambda_t} \right)
			+  \tilde B\Bigl( \lambda_t e^{\frac{\tau^2}{2}} \Bigr)
		\quad\text{for $\tau\in[0,t]$}.
\end{equation}
Now, assume that $\lambda_t$ obeys \eqref{E:lambda-t-condition}, namely
\begin{equation} \label{mar10}
	\int_{0}^{t} B\left( \frac{f_t(\tau)}{\lambda_t} \right) e^{-\frac{\tau^2}{2}}\,\d \tau
		= \sqrt{2\pi}.
\end{equation}
Observe that $\lambda_t$ is uniquely defined for each $t\in(0,\infty)$ by the
monotonicity of the function $B\circ b^{-1}$.
Combining \eqref{mar9} with \eqref{mar10}, one obtains that
\begin{equation} \label{mar11}
	\int_{0}^{t} f_t(\tau)\,\d \tau
		= \sqrt{2\pi}
		+ \int_{0}^{t} \tilde B\Bigl( \lambda_t e^{\frac{\tau^2}{2}} \Bigr)
				e^{-\frac{\tau^2}{2}}\,\d \tau,
\end{equation}
whence, via \eqref{mar8} and \eqref{mar11},
\begin{align*}
	 \sup\left\{ \int_{0}^{t} f(\tau)\,\d \tau:
		\int_{0}^{t} B\left( \frac{f(\tau)}{\lambda_t} \right)
			e^{-\frac{\tau^2}{2}}\,\d \tau\le \sqrt{2\pi}
	\right\}
		= \sqrt{2\pi}
		+ \int_{0}^{t} \tilde B\Bigl( \lambda_t e^{\frac{\tau^2}{2}} \Bigr)
				e^{-\frac{\tau^2}{2}}\,\d \tau
		= \int_{0}^{t} f_t(\tau)\,\d \tau.
\end{align*}
Therefore
\begin{align*}
	\normIB
		& = \sup\left\{ \int_{0}^{t} \frac{f(\tau)}{\lambda_t}\,\d \tau:
			\int_{0}^{t} B\left( \frac{f(\tau)}{\lambda_t} \right)
				e^{-\frac{\tau^2}{2}}\,\d \tau\le \sqrt{2\pi}
			\right\}
			\\
		& = \frac{1}{\lambda_t} \int_{0}^{t} f_t(\tau)\,\d \tau
			= \int_{0}^{t} b^{-1} \Bigl(\lambda_t e^{\frac{\tau^2}{2}}\Bigr)\,\d \tau,
\end{align*}
and \eqref{E:I-norm-general} follows.
\end{proof}

\begin{lemma}
\label{L:lambda-t-asymp}
Let $B$ be a Young function of the form \eqref{BN}
for some $\beta \in (0,2]$ and $N>0$. Assume that $\lambda_t$
satisfies \eqref{E:lambda-t-condition}.
Then
\begin{equation}
\label{E:log-lambda-t}
	\lambda_t = c_\beta
	\begin{cases}
		t^{1-\frac{2}{\beta}}+ \cdots
			& \text{if $\beta\in(0,2)$}
			\\
		\frac{1}{\log t} + \cdots
			& \text{if $\beta=2$}
	\end{cases}
	\quad\text{as $t\to\infty$,}
\end{equation}
where
\begin{equation*}
	c_{\beta} = 2^{\frac{1}{\beta}-\frac{1}{2}} \sqrt{\pi}
	\begin{cases}
		2-\beta
			& \text{if $\beta\in(0,2)$}
			\\
		2
			& \text{if $\beta=2$}.
	\end{cases}
\end{equation*}
\end{lemma}

\begin{proof}
One has that
\begin{equation*}
	B(t)
		= Ne^{t^\beta}
		= \ib t^{1-\beta}b(t)
\end{equation*}
for large $t$.  Hence, by Lemma~\ref{L:a-invers} and equation \eqref{E:power},
\begin{equation} \label{may2}
	B(b^{-1}(t))
 	= \ib t(\log t)^{\frac{1}{\beta}-1}
			+ \frac{(1-\beta)^2}{\beta^3} t (\log t)^{\ib-2}\log\log t
			+ \cdots
			\quad\text{as $t\to\infty$.}
\end{equation}
If $\beta\neq 1$, the second addend on the right-hand side of equation
\eqref{may2} is strictly positive. Consequently, there exists
$t_0\in(1,\infty)$ such that
\begin{equation} \label{may1}
	B(b^{-1}(t)) \ge \frac{t}{\beta} (\log t)^{\frac{1}{\beta}-1}
		\quad\text{for $t>t_0$.}
\end{equation}
If $\beta=1$, then $B=b$ near infinity, and \eqref{may1} holds, as
equality, as well.

Now, note that  $\lambda_t$ is a decreasing function of $t$,  by the
monotonicity of the function $B\circ b^{-1}$.
We claim that
\begin{equation}\label{dec35}
	\lim_{t\to\infty}\lambda_t = 0.
\end{equation}
To see this, assume, by contradiction, that
$\lim_{t\to\infty}\lambda_t=\lambda$ for some $\lambda>0$.  Choose $\tau_0>0$
so that $\lambda e^{\frac{\tau^2}{2}}\ge t_0$ for $\tau>\tau_0$. Then, letting
$t\to\infty$ in \eqref{E:lambda-t-condition}, yields, by Fatou's lemma,
\begin{align*}
	\sqrt{2\pi}
		\ge \int_{0}^{\infty} B\left( b^{-1}\Bigl( \lambda e^\frac{\tau^2}{2} \Bigr) \right)
			e^{-\frac{\tau^2}{2}}\,\d \tau
		\ge \int_{\tau_0}^{\infty} B\left( b^{-1}\Bigl(\lambda e^\frac{\tau^2}{2}\Bigr)\right)
		e^{-\frac{\tau^2}{2}}\,\d \tau
		\ge \frac{\lambda}{\beta}
			\int_{\tau_0}^{\infty} \left(\frac{\tau^2}{2}+\log\lambda\right)^{\ib-1}\,\d \tau.
\end{align*}
This is impossible, since the last integral diverges, inasmuch as
$2/\beta-2\ge-1$. Equation \eqref{dec35} is therefore established.

Now, fix $t\geq t_0$ so large  that $\lambda_{t}<1$,
and set
\begin{equation} \label{may3}
	\tau(t)=\sqrt{2\log\frac{t_0}{\lambda_t}}.
\end{equation}
If $t$ is such that
\begin{equation}\label{dec36}
\tau(t)<t,
\end{equation}
 then, owing to equation \eqref{may1},
\begin{align} \label{may5}
	\begin{split}
	\sqrt{2\pi}
		 & = \int_{0}^{t}
			B\left( b^{-1}\Bigl( \lambda_t e^\frac{\tau^2}{2} \Bigr) \right)
				e^{-\frac{\tau^2}{2}}\,\d \tau
				\\
			& \ge \int_{\tau(t)}^{t}
			B\left( b^{-1}\Bigl( \lambda_t e^\frac{\tau^2}{2} \Bigr) \right)
				e^{-\frac{\tau^2}{2}}\,\d \tau
				 \ge \frac{\lambda_t}{\beta}
			\int_{\tau(t)}^{t}\left(\frac{\tau^2}{2}+\log\lambda_t \right)^{\ib-1}\d\tau.
	\end{split}
\end{align}
Hence, by the change of variables $\tau =r\sigma(t)$, where
\begin{equation} \label{E:def-of-tau}
	\sigma(t) = \sqrt{2\log\frac{1}{\lambda_t}},
\end{equation}
one obtains that
\begin{align} \label{may2a}
	\begin{split}
	\int_{\tau(t)}^{t}\left(\frac{\tau^2}{2}+\log\lambda_t \right)^{\ib-1}\,\d \tau
		& = \int_{\tau(t)}^{t}\left(\frac{\tau^2}{2}-\frac{\sigma(t)^2}{2} \right)^{\ib-1}\,\d \tau
			\\
		& = 2^{1-\ib} \sigma(t)^{\frac{2}{\beta}-1}\int_{\tau(t)/\sigma(t)}^{t/\sigma(t)} (r^2-1)^{\ib-1}\,\d r
			\\
		& = 2^{1-\ib} \sigma(t)^{\frac{2}{\beta}-1}
			\left[ \Psi_{\ib-1}\left( \frac{t}{\sigma(t)} \right)
				- \Psi_{\ib-1}\left( \frac{\tau(t)}{\sigma(t)} \right)\right].
	\end{split}
\end{align}
Here, $\Psi_{\ib-1}$ denotes the function defined   as in \eqref{E:Psi}, with
$\lowerlimit=1$.  If, in addition to  \eqref{dec36}, we assume that
\begin{equation}\label{dec37}
	\lim_{t\to\infty} \frac{t}{\sigma(t)}
		= \infty,
\end{equation}
then, since $\lim_{t\to\infty} \tau(t)/\sigma(t)= 1$,  Lemma~\ref{L:Psi} entails that
\begin{equation} \label{may6}
	\int_{\tau(t)}^{t}\left(\frac{\tau^2}{2}+\log\lambda_t \right)^{\ib-1}\,\d \tau
		= 2^{1-\ib} \begin{cases}
			\frac{\beta}{2-\beta} t^{\iib-1}
				+ \cdots
				&\text{if $\beta\in(0,2)$}
				\\
			\log t - \log \sigma(t)
				+ \cdots
				&\text{if $\beta=2$}
		\end{cases}
	\quad\text{as $t\to\infty$}.
\end{equation}
Next, fix $\varepsilon>0$ and observe that equation \eqref{may2} implies that
\begin{equation}\label{dec38}
	B(b^{-1}(t)) \le \frac{1+\varepsilon}{\beta} t (\log t)^{\ib-1}
		\quad\text{for $t>t_0$.}
\end{equation}
We may assume, without loss of generality,  that inequality \eqref{dec38} holds
with the same $t_0$ as in~\eqref{may1}, by choosing a larger value of $t_0$, if
necessary.  Since $B$ is a Young function, we have that $B(t)\le tb(t)$ for
$t>0$, whence
\begin{equation*}
	B(b^{-1}(t))\le tb^{-1}(t)
		\quad\text{for $t>0$}.
\end{equation*}
Therefore, if $\tau(t)<t$, then
\begin{align} \label{may1a}
	\begin{split}
	\sqrt{2\pi}
		& = \int_{0}^{\tau(t)}
			B\left( b^{-1}\Bigl( \lambda_t e^\frac{\tau^2}{2} \Bigr) \right)
				e^{-\frac{\tau^2}{2}}\,\d \tau
			+ \int_{\tau(t)}^{t}
			B\left( b^{-1}\Bigl( \lambda_t e^\frac{\tau^2}{2} \Bigr) \right)
				e^{-\frac{\tau^2}{2}}\,\d \tau
				\\
		& \le \lambda_t \int_0^{\tau(t)} b^{-1}\left(\lambda_te^{\frac{\tau^2}{2}}\right)\d \tau
			+ \frac{1+\varepsilon}{\beta} \lambda_t \int_{\tau(t)}^{t}
					\left(\frac{\tau^2}{2}+\log\lambda_t \right)^{\ib-1}\d \tau
				\\
		& \le \lambda_t\,\tau(t)\, b^{-1}\left(\lambda_te^{\frac{\tau(t)^2}{2}}\right)
				+ \frac{1+\varepsilon}{\beta} \lambda_t \int_{\tau(t)}^{t}
					\left(\frac{\tau^2}{2}+\log\lambda_t \right)^{\ib-1}\d \tau
				\\
		& = \lambda_t\,\tau(t)\, b^{-1}(t_0)
				+ \frac{1+\varepsilon}{\beta} \lambda_t \int_{\tau(t)}^{t}
					\left(\frac{\tau^2}{2}+\log\lambda_t \right)^{\ib-1}\d \tau
			\quad\text{for $t\ge t_0$,}
	\end{split}
\end{align}
whereas, if $\tau(t)\ge t$, then
\begin{equation} \label{may1b}
	\sqrt{2\pi}
		\le \int_{0}^{\tau(t)}
			B\left( b^{-1}\Bigl( \lambda_t e^\frac{\tau^2}{2} \Bigr) \right)
				e^{-\frac{\tau^2}{2}}\,\d \tau
		\le \lambda_t\,\tau(t)\, b^{-1}(t_0)
		\quad\text{for $t\ge t_0$.}
\end{equation}
For later use,  observe also that, by the very definition \eqref{may3} of $\tau(t)$,
\begin{equation}\label{E:product-to-zero}
	\lim_{t\to\infty}\lambda_t\, \tau(t)=0.
\end{equation}
Let $\beta\in(0,2)$. In order to prove~\eqref{E:log-lambda-t}, we have to show that
\begin{equation}\label{E:limit-lambda-t}
    \lim_{t\to\infty}\frac{\lambda_t}{c_{\beta}t^{1-\frac{2}{\beta}}}=1.
\end{equation}
Assume, by  contradiction, that~\eqref{E:limit-lambda-t} fails. Then there
exist $\delta>0$ and a sequence $\{t_k\}$ such that $\lim_{k\to\infty}
t_k=\infty$, and  either
\begin{equation}\label{E:4.44}
	\lambda_{t_k}\ge(1+\delta)c_{\beta}t_k^{1-\frac{2}{\beta}},
\end{equation}
or
\begin{equation}\label{E:4.45}
	\lambda_{t_k}\le(1-\delta)c_{\beta}t_k^{1-\frac{2}{\beta}}
\end{equation}
for $k\in\N$.  Firstly, suppose that~\eqref{E:4.44} is in force. Observe that
the sequences $\{\tau(t_k)\}$ and $\{\sigma(t_k)\}$ defined as in~\eqref{may3}
and~\eqref{E:def-of-tau}, respectively, satisfy
\begin{equation} \label{E:4.46}
\sigma(t_k) < \tau(t_k) \le \sqrt{2\left(\frac{2}{\beta}-1\right)\log t_k+2\log\frac{t_0}{(1+\delta)c_{\beta}}}.
\end{equation}
Equation \eqref{E:4.46} ensures that condition \eqref{dec36} is fulfilled with $t=t_k$ for large $k$, namely
\begin{equation}\label{dec200}
\tau(t_k)<t_k,
\end{equation}
 and that the limit in \eqref{dec37} holds when evaluated on the sequence $\{t_k\}$, \ie
\begin{equation}\label{dec201}
	\lim_{k\to\infty}\frac{t_k}{\sigma(t_k)}=\infty.
\end{equation}
Hence, owing  to equations \eqref{may5},~\eqref{may6} and~\eqref{E:4.46},
\begin{align*}
	\sqrt{2\pi}
	& \ge \frac{\lambda_{t_k}}{\beta}\int_{\tau(t_k)}^{t_k}\left(\frac{\tau^2}{2}+\log \lambda_t\right)^{\frac{1}{\beta}-1}\,\d \tau
	= \frac{\lambda_{t_k}}{\beta}2^{1-\frac{1}{\beta}}\frac{\beta}{2-\beta}t_k^{\frac{2}{\beta}-1}+\cdots
		\\
	& \ge (1+\delta)c_{\beta}\frac{2^{1-\frac{1}{\beta}}}{2-\beta}+\cdots
		= (1+\delta) \sqrt{2\pi}+\cdots
\end{align*}
for large $k$, a contradiction.

Secondly, suppose that~\eqref{E:4.45} holds. Then inequality \eqref{dec200} holds,
provided that $k$ is large enough,  otherwise,
by \eqref{may1b} (on taking  a subsequence, if necessary),
\begin{equation} \label{jun9}
	\sqrt{2\pi}
		\le \lambda_{t_k}\,\tau(t_k)\, b^{-1}(t_0),
\end{equation}
which contradicts~\eqref{E:product-to-zero}. If there exist a subsequence of $\{t_k\}$,  still  denoted by $\{t_k\}$, and a constant $c\in[1,\infty)$ such that
$t_k/\tau(t_k)\to c$, then, by \eqref{may2a} and \eqref{may1a},
\begin{align} \label{jun10}
	\begin{split}
	\sqrt{2\pi}
		& \le \lambda_{t_k}\,\tau(t_k)\,b^{-1}(t_0)
			+ \frac{1+\varepsilon}{\beta}\lambda_{t_k}\int_{\tau(t_k)}^{t_k}\left(\frac{\tau^2}{2}+\lambda_{t_k}\right)^{\frac{1}{\beta}-1}\,\d \tau
			\\
		& = \lambda_{t_k}\,\tau(t_k)\,b^{-1}(t_0)
			+ \frac{1+\varepsilon}{\beta}2^{1-\ib}e^{-\frac{\sigma(t_k)^2}{2}}\sigma(t_k)^{\iib-1}
				\left[
					\Psi_{\frac{1}{\beta}-1}\left(\frac{t_k}{\sigma(t_k)}\right)
					-\Psi_{\frac{1}{\beta}-1}\left(\frac{\tau(t_k)}{\sigma(t_k)}\right)
				\right]
	\end{split}
\end{align}
for large $k$.
Clearly,
\begin{equation*}
	\lim_{k\to\infty}
		\left[
			\Psi_{\frac{1}{\beta}-1}\left(\frac{t_k}{\sigma(t_k)}\right)
			- \Psi_{\frac{1}{\beta}-1}\left(\frac{\tau(t_k)}{\sigma(t_k)}\right)
		\right]
	= \Psi_{\frac{1}{\beta}-1}(c),
\end{equation*}
since $\lim_{k\to\infty}\tau(t_k)/\sigma(t_k)=1$. From equation~\eqref{E:product-to-zero}
and the fact that $\lim_{k\to\infty}\sigma(t_k)=\infty$, we conclude that the right-hand side
of~\eqref{jun10} tends to zero as $k\to\infty$, a contradiction.

It remains to consider the case when, up to subsequences,
$t_k/\tau(t_k)\to\infty$. This implies that equation \eqref{dec201} holds  as well. Thus,
by \eqref{may1a} and \eqref{may6},
\begin{equation*}
	\sqrt{2\pi}
		\le \lambda_{t_k}\,\tau(t_k)\, b^{-1}(t_0)
			+ \frac{1+\varepsilon}{\beta}2^{1-\ib}
				\frac{\beta}{2-\beta}\lambda_{t_k}\,t_k^{\frac{2}{\beta}-1}
			+ \cdots
		\quad\text{as $k\to\infty$},
\end{equation*}
whence, owing to~\eqref{E:4.45},
\begin{equation}\label{dec40}
	\sqrt{2\pi}
		\le \lambda_{t_k}\,\tau(t_k)\, b^{-1}(t_0)
			+ (1+\varepsilon)(1-\delta)\sqrt{2\pi} + \cdots
		\quad\text{as $k\to\infty$}.
\end{equation}
This  again leads to a contradiction, provided that  $\varepsilon$ is chosen
so small  that $1+\varepsilon<\frac{1}{1-\delta}$, since   the first addend on
the right-hand side of equation \eqref{dec40} tends to $0$ as $k \to\infty$,
thanks to~\eqref{E:product-to-zero}.

Assume now that $\beta=2$. In order to prove~\eqref{E:log-lambda-t}, we need to show that
\begin{equation}\label{E:limit-lambda-t-beta=2}
    \lim_{t\to\infty}\frac{\lambda_t}{\frac{c_2}{\log t}}=1.
\end{equation}
Suppose, by contradiction, that~\eqref{E:limit-lambda-t-beta=2} fails. Then there exist
$\delta\in(0,1)$ and a  sequence $\{t_k\}$ such that $\lim_{k \to \infty} t_k=\infty$ and either
\begin{equation} \label{may4a}
	\lambda_{t_k}\ge(1+\delta)\frac{c_2}{\log t_k}
\end{equation}
or
\begin{equation} \label{may4b}
	\lambda_{t_k}\le(1-\delta)\frac{c_2}{\log t_k}
\end{equation}
for $k\in\N$.
Assume that \eqref{may4a} is satisfied.
Observe that $\tau(t_k)$ and $\sigma(t_k)$ obey
\begin{equation*}
	\sigma(t_k) < \tau(t_k) \le \sqrt{2\log\log t_k+2\log\frac{t_0}{(1+\delta)c_2}}
\end{equation*}
for large $k$. Hence, $\tau(t_k)<t_k$ for large $k$, and
$t_k/\sigma(t_k)\to\infty$ and $\log t_k/\log \sigma(t_k)\to\infty$ as
$k\to\infty$. Consequently, equations~\eqref{may5}, \eqref{may6} and
\eqref{may4a} imply that
\begin{align*}
	\sqrt{2\pi}
		 \ge \frac{\lambda_{t_k}}{2}\sqrt{2}\left(\log t_k-\log \sigma(t_k)\right) + \cdots
			= \frac{1}{\sqrt{2}}\lambda_{t_k} \log t_k + \cdots
			\ge (1+\delta)\sqrt{2\pi} + \cdots \quad\text{as $k\to\infty$,}
\end{align*}
a contradiction.

Suppose next that \eqref{may4b} is in force.  Notice that
inequalities \eqref{jun9} and \eqref{jun10} also hold if $\beta=2$. The same
argument as in the case when $\beta \in (0, 2)$ then rules out all cases where
there does not exist a~subsequence of $\{t_k\}$, called again $\{t_k\}$, such
that
\begin{equation*}
	\lim_{k\to\infty}\frac{t_k}{\tau(t_k)} = \infty.
\end{equation*}
We can thereby assume that such a subsequence exists. Thanks to the fact that $\sigma(t_k)>1$ if $k$ large enough, we infer from \eqref{may6} and \eqref{may1a} that
\begin{align*}
	\sqrt{2\pi}
		& \le \lambda_{t_k}\,\tau(t_k)\, b^{-1}(t_0)
			+ \frac{1+\varepsilon}{2}\sqrt{2}\lambda_{t_k}(\log t_{k} - \log \sigma(t_{k}))
			+ \cdots
			\\
		& \le \lambda_{t_k}\,\tau(t_k)\, b^{-1}(t_0)
			+ \frac{1+\varepsilon}{2}\sqrt{2}\lambda_{t_k}\log t_{k}
			+ \cdots
			\\
		& = \lambda_{t_k}\,\tau(t_k)\, b^{-1}(t_0)
			+ (1+\varepsilon)(1-\delta)\sqrt{2\pi} + \cdots
			\quad\text{as $k\to\infty$}.
\end{align*}
This is again a contradiction, if $\varepsilon$
is chosen  in such a way that $1+\varepsilon<\frac{1}{1-\delta}$, since the first addend on the rightmost side approaches zero as $k \to \infty$.
\end{proof}

\begin{lemma}
\label{L:I-norm-sharp}
Let $B$ be a Young function of the form \eqref{BN} for some $\beta \in (0,2]$
and $N>0$, and let $\Phi$ and $I$ be the functions given by \eqref{E:Phi-def}
and \eqref{I}, respectively.  Then
\begin{equation*}
	\normIB
		=  2^{-\ib} \frac{\beta}{2+\beta} t^{\iib+1}
			+ 2^{-\ib}
			\begin{cases}
				\displaystyle
				- \frac{2}{2-\beta}t^{\iib-1}\log t
				+ c_{\beta,N} t^{\iib-1}
				+ \cdots
					& \text{if $\beta\in(0,2)$}
					\\[\bigskipamount]
				\displaystyle
				- \frac{1}{2}(\log t)^2
				- \log t \log\log t
				+ \cdots
					& \text{if $\beta=2$}
			\end{cases}
\end{equation*}
as $t\to\infty$, where $c_{\beta,N}$ is a  constant depending on $\beta$
and $N$.
\end{lemma}

\begin{proof}
Fix
$\varepsilon\in(0,1)$. By~Lemma~\ref{L:a-invers}, there exists
$\tau_0>0$ such that
\begin{equation}\label{E:V-new}
	b^{-1}(\tau) \le (1+\varepsilon) (\log\tau)^\ib,
\end{equation}
\begin{equation}\label{E:b-invers-upper}
	b^{-1}(\tau)
		\le (\log \tau)^\ib
			+ \frac{1-\beta}{\beta^2} (\log \tau)^{\ib-1}\log\log \tau
			+ K_\varepsilon^+ (\log \tau)^{\ib-1}
\end{equation}	
and
\begin{equation}\label{E:b-invers-lower}
	b^{-1}(\tau)
		\ge (\log \tau)^\ib
			+ \frac{1-\beta}{\beta^2} (\log \tau)^{\ib-1}\log\log \tau
			+ K_\varepsilon^- (\log \tau)^{\ib-1}
\end{equation}
for $\tau>\tau_0$, where
\begin{equation*}
	K_\varepsilon^+ =-\frac{\log\beta N}{\beta} +\varepsilon
	\quad\text{and}\quad
	K_\varepsilon^- = - \frac{\log\beta N}{\beta} - \varepsilon.
\end{equation*}
Lemma~\ref{L:lambda-t-asymp} tells us that  the function $t\mapsto\lambda_t$, $t>0$,
defined in Lemma~\ref{L:I-norm-general}, decreases to zero
as $t\to\infty$.  Hence, there exists $t_0>0$ such that $\lambda_t<1$ for
$t>t_0$. Let  $\sigma(t)$ be the function defined  by~\eqref{E:def-of-tau} for
$t>t_0$. Then $\sigma(t)$ is increasing on $(t_0,\infty)$ and
\begin{equation}\label{E:lambda-in-terms-of-tau}
	\lambda_t = e^{-\frac{\sigma(t)^2}{2}}
		\quad\text{for $t>t_0$.}
\end{equation}
Lemma~\ref{L:lambda-t-asymp} also tells us that
\begin{equation}\label{E:estimate-tau-less-than-2}
	\sigma(t)=
		\begin{cases}
			\sqrt{2\left(\frac{2}{\beta}-1\right)\log t} + \cdots
				& \text{if $\beta\in(0,2)$}
				\\
			\sqrt{2\log\log t} + \cdots
				& \text{if $\beta=2$}
		\end{cases}
	\quad\text{as $t\to\infty$}.
\end{equation}
Consequently,
\begin{equation}\label{E:tau-decay}
	\lim_{t \to \infty}\frac{t}{\sigma(t)}=\infty\,.
\end{equation}
Next, choose $C>1$ fulfilling the inequality
\begin{equation}\label{E:definition-of-C}
	e^{(C^2-1)\frac{\sigma(t)^2}{2}} > \tau_0
		\quad\text{for $t>t_0$}.
\end{equation}
Thanks to~\eqref{E:tau-decay}, there exists  $t_1>t_0$
such that $t>C\sigma(t)$ for $t>t_1$. On setting
\begin{equation}\label{E:definition-of-I1}
	I_1(t) = \int_0^{C\sigma(t)}
		b^{-1}\Bigl(\lambda_t e^{\frac{\tau^2}{2}}\Bigr)\,\d \tau
\end{equation}
and
\begin{equation}\label{E:definition-of-I2}
	I_2(t) = \int_{C\sigma(t)}^t
		b^{-1}\Bigl(\lambda_t e^{\frac{\tau^2}{2}}\Bigr)\,\d \tau
\end{equation}
for $t>t_1$, equation  \eqref{E:I-norm-general} can be rewritten as
\begin{equation}\label{E:splitting-of-Is}
	\normIB
		= I_1(t) + I_2(t)
		\quad\text{for $t>t_1$}.
\end{equation}
We begin with an estimate for $I_1$. By~\eqref{E:definition-of-I1}
and the monotonicity of $b^{-1}$,
\begin{equation*}
	I_1(t)
		\le C \sigma(t)\,b^{-1}\Bigl(\lambda_t\,e^{\frac{C^2\sigma(t)^2}{2}}\Bigr)
		\quad\text{for $t>t_1$}.
\end{equation*}
By equation \eqref{E:lambda-in-terms-of-tau}, this inequality ensures that
\begin{equation}\label{E:I1-intermediate}
	I_1(t)
		\le C\sigma(t)\,b^{-1}\Bigl(e^{(C^2-1)\frac{\sigma(t)^2}{2}}\Bigr)	
		\quad\text{for $t>t_1$}.
\end{equation}
Thus, by~\eqref{E:I1-intermediate}, \eqref{E:definition-of-C} and ~\eqref{E:V-new},
\begin{equation}\label{E:new2}
	I_1(t) \le K \sigma(t)^{\iib+1}
		\quad\text{for $t>t_1$},
\end{equation}
where
\begin{equation*}
	K = (1+\varepsilon)C\left(\frac{C^2-1}{2}\right)^\ib.
\end{equation*}
Therefore, from ~\eqref{E:estimate-tau-less-than-2} and~\eqref{E:new2} we obtain that
\begin{equation}\label{E:estimate-of-I1}
	0 \le I_1(t) \le K 2^{\frac1{\beta}+\frac12}
		\begin{cases}
			\left(\left(\frac{2}{\beta}-1\right)\log t\right)^{\frac1{\beta}+\frac12}
				+\cdots
				&\text{if $\beta\in(0,2)$}
				\\
			\log\log t + \cdots
				&\text{if $\beta=2$}
		\end{cases}
	\quad\text{as $t\to\infty$.}
\end{equation}
We  now deal with $I_2$. Equations ~\eqref{E:lambda-in-terms-of-tau}
and~\eqref{E:definition-of-I2} entail that $I_2$ can be expressed as
\begin{equation} 
	I_2(t) = \int_{C\sigma(t)}^t
		b^{-1}\Bigl(e^{\frac{\tau^2}{2}-\frac{\sigma(t)^2}{2}}\Bigr)\,\d \tau
		\quad\text{for $t>t_1$.}
\end{equation}
From inequality ~\eqref{E:definition-of-C} one has that
\begin{equation} 
	e^{\frac{\tau^2}{2}-\frac{\sigma(t)^2}{2}} > \tau_0
		\quad\text{for $\tau>C\sigma (t)$, }
\end{equation}
provided that $t>t_1$.
Set
\begin{equation} 
	I_{21}(t)
		= \int_{C\sigma(t)}^t \left(\frac{\tau^2}{2}-\frac{\sigma(t)^2}{2}\right)^{\ib}\d\tau,
\end{equation}
\begin{equation} 
	I_{22}(t)
		= \int_{C\sigma(t)}^t\left(\frac{\tau^2}{2}-\frac{\sigma(t)^2}{2}\right)^{\ib-1}
			\log\left(\frac{\tau^2}{2}-\frac{\sigma(t)^2}{2}\right)\d \tau
\end{equation}
and
\begin{equation} 
	I_{23}(t)
		= \int_{C\sigma(t)}^t\left(\frac{\tau^2}{2}-\frac{\sigma(t)^2}{2}\right)^{\ib-1}\d \tau,
\end{equation}
for $t>t_1$. By inequality \eqref{E:b-invers-upper},
\begin{align}\label{dec15}
	\begin{split}
		I_{2}(t)
			\le I_{21}(t)
				+ \frac{1-\beta}{\beta^2} I_{22}(t)
				+ K_{\varepsilon}^{+}I_{23}(t)
	\end{split}
\end{align}
and, by \eqref{E:b-invers-lower}
\begin{align}\label{dec16}
	\begin{split}
		I_{2}(t)
			\ge I_{21}(t)
				+ \frac{1-\beta}{\beta^2}I_{22}(t)
				+ K_{\varepsilon}^{-}I_{23}(t)
	\end{split}
\end{align}
for $t>t_1$.

Let us focus  on $I_{21}$. By a change of variables,
\begin{align} 
	\begin{split}
		I_{21}(t)
			& = \int_{C}^{\frac{t}{\sigma(t)}}
				\left( \frac{\sigma(t)^2}{2} r^2 - \frac{\sigma(t)^2}{2} \right)^\ib\!\sigma(t)\,\d r
				= 2^{-\ib} \sigma(t)^{\iib+1}
				\int_{C}^{\frac{t}{\sigma(t)}} (r^2-1)^\ib \,\d r
				\\
			& = 2^{-\ib} \sigma(t)^{\iib+1}
				\Psi_{\ib}\left( \frac{t}{\sigma(t)} \right)
				\quad\text{for $t>t_1$},
	\end{split}
\end{align}
where $\Psi_{\ib}$ is defined as in  \eqref{E:Psi}, with $\lowerlimit=C$.
Thanks to equation \eqref{E:tau-decay}, we may use Lemma~\ref{L:Psi} to infer
what follows: if $\beta\in(0,2)$, then
\begin{align*}
	\Psi_{\frac1\beta}\left( \frac{t}{\sigma(t)} \right)
		 = \frac{\beta}{2+\beta} \left(\frac{t}{\sigma(t)} \right)^{\iib+1}
			- \dfrac{1}{2-\beta} \left(\dfrac{t}{\sigma(t)} \right)^{\iib-1}
			+
			\begin{cases}
			\frac{1}{2\beta}\frac{1-\beta}{2-3\beta}
				\left(\frac{t}{\sigma(t)} \right)^{\iib-3} + \cdots
				& \text{if $\beta\in(0,\frac{2}{3})$}
				\\
			\frac{3}{8}\log\frac{t}{\sigma(t)} + \cdots
				& \text{if $\beta=\frac{2}{3}$}
				\\
			c + \cdots
				& \text{if $\beta\in(\frac{2}{3},2)$}
			\end{cases}
\end{align*}
as $t\to\infty$, whence
\begin{align} 
	\begin{split}
		I_{21}(t)
			 = 2^{-\ib} \frac{\beta}{2+\beta} t^{\iib+1}
				-	2^{-\ib}\frac{1}{2-\beta}t^{\iib-1}\sigma(t)^{2}
		 + 2^{-\ib}
					\begin{cases}
						\frac{1}{2\beta}\frac{1-\beta}{2-3\beta}\, t^{\iib-3}\sigma(t)^{4} + \cdots
						& \text{if $\beta\in(0,\frac{2}{3})$}
									\\
								\frac{3}{8}\sigma(t)^{\iib+1} \log t + \cdots
						& \text{if $\beta=\frac{2}{3}$}
									\\
						c \sigma(t)^{\iib+1} + \cdots
						& \text{if $\beta\in(\frac{2}{3},2)$}
					\end{cases}
	\end{split}
\end{align}
as $t\to\infty$; if $\beta=2$, then
\begin{equation*}
	\Psi_{\frac1\beta}\left( \frac{t}{\sigma(t)} \right)
		= \frac{1}{2}\left( \frac{t}{\sigma(t)} \right)^2-\frac{1}{2}\log\frac{t}{\sigma(t)}
			+ c + \cdots
		\quad\text{as $t\to\infty$},
\end{equation*}
whence
\begin{equation*}
	I_{21}(t)
		= \frac{1}{2\sqrt{2}}t^2
		- \frac{1}{2\sqrt{2}}\sigma(t)^2\log t
		+ \frac{1}{2\sqrt{2}}\sigma(t)^2\log \sigma(t)
		+ \cdots
	\quad\text{as $t\to\infty$.}
\end{equation*}
Also, coupling equation ~\eqref{E:def-of-tau} with Lemma~\ref{L:lambda-t-asymp}
enables us to deduce that
\begin{equation*}
	\sigma(t)^2=
		\begin{cases}
			2\left(\frac{2}{\beta}-1\right)\log t - 2\log c_{\beta} + \cdots
				& \text{if $\beta\in(0,2)$}
					\\
			2\log\log t - 2 \log c_2 + \cdots
				& \text{if $\beta=2$}
		\end{cases}
	\quad\text{as $t\to\infty$},
\end{equation*}
whence
\begin{equation}\label{E:I21-final}
I_{21}(t)
	= 2^{-\ib} \frac{\beta}{2+\beta} t^{\iib+1}
		+	2^{-\ib}
		\begin{cases}
			- \iib t^{\iib-1}\log t
			+ \frac{2}{2-\beta} \log c_{\beta} t^{\iib-1} + \cdots
					& \text{if $\beta\in(0,2)$}
						\\
			- \log t\log\log t
			+  \log c_2 \log t + \cdots
					& \text{if $\beta=2$}
		\end{cases}
\end{equation}
as $t\to\infty$.
Let us next deal with $I_{22}(t)$.
We have that
\begin{align} \label{E:I-22}
	\begin{split}
	I_{22}(t)
		& = \int_{C\sigma(t)}^{t} \left(\frac{\tau^2}{2}-\frac{\sigma(t)^2}{2}\right)^{\ib-1}
			 \log\left(\frac{\tau^2}{2}-\frac{\sigma(t)^2}{2}\right)\d \tau
			\\
		& = 2^{1-\ib}\sigma(t)^{\iib-1}
			\left[
				\log\frac{\sigma(t)^2}{2}\,\Psi_{\ib-1}\left(\frac{t}{\sigma(t)}\right)
				+ \Upsilon_{\ib-1}\left( \frac{t}{\sigma(t)}\right)
			\right]
		 = I_{221}(t) + I_{222}(t)
	\end{split}
\end{align}
for $t>t_1$, where the functions $\Psi_{\ib-1}$ and $\Upsilon_{\ib-1}$ are
defined as in \eqref{E:Psi} and \eqref{E:Upsilon}, with  $d=C$, and  where we
have set
\begin{equation*}
	I_{221}(t)
		= 2^{1-\ib}\sigma(t)^{\iib-1}\log\frac{\sigma(t)^2}{2}
			\Psi_{\ib-1}\left(\frac{t}{\sigma(t)}\right),
\end{equation*}
and
\begin{equation*}
	I_{222}(t)
		= 2^{1-\ib}\sigma(t)^{\iib-1}
		\Upsilon_{\ib-1}\left(\frac{t}{\sigma(t)}\right)
\end{equation*}
for $t>t_1$.
As far as  $I_{221}$ is concerned, thanks to Lemma~\ref{L:Psi},
\begin{equation*}
	\Psi_{\ib-1}\left( \frac{t}{\sigma(t)}\right)
		= \frac{\beta}{2-\beta}\left( \frac{t}{\sigma(t)}\right)^{\iib-1}
			- \begin{cases}
		\frac{\beta-1}{3\beta-2}\left(\frac{t}{\sigma(t)}\right)^{\iib-3}
			+ \cdots &\text{if $\beta\in\left(0,\frac23\right)$}
				\\
		\frac12\log \frac{t}{\sigma\left(t\right)}
		 +\cdots
			& \text{if $\beta=\frac23$}
			\\
			c+\cdots
			& \text{if $\beta\in\left(\frac23,2\right)$}
		\end{cases}
			\quad\text{as $t\to\infty$.}
\end{equation*}
Hence, if  $\beta\in(0,2)$, then
\begin{equation*}
	\sigma(t)^{\iib-1}\Psi_{\ib-1}\left(\frac{t}{\sigma(t)}\right)
		= \frac{\beta}{2-\beta}t^{\iib-1}
		- \begin{cases}
				\frac{\beta-1}{3\beta-2}t^{\iib-3}\sigma(t)^2
					+ \cdots &\text{if $\beta\in\left(0,\frac23\right)$}
					\\
				\frac12 \log t
					+\cdots
					& \text{if $\beta=\frac23$}
					\\
				c + \cdots
					& \text{if $\beta\in\left(\frac23,2\right)$}
			\end{cases}
				\quad\text{as $t\to\infty$},
\end{equation*}
and, if $\beta=2$, then
\begin{equation*}
	\Psi_{-\frac{1}{2}}\left(\frac{t}{\sigma(t)}\right)
		= \log t - \log \sigma(t) + c + \cdots
		\quad\text{as $t\to\infty$}.
\end{equation*}
Note that $\sigma(t)^{\iib-1}=1$ in the latter case. Equation
~\eqref{E:def-of-tau} and Lemma~\ref{L:lambda-t-asymp} imply that
\begin{equation*}
	\log\frac{\sigma(t)^2}{2} =
			\begin{cases}
				\log\log t + \log \left(\iib-1\right) + \cdots
					& \text{if $\beta\in(0,2)$}
						\\[\medskipamount]
				\log\log\log t - \frac{\log c_2}{\log\log t} + \cdots
					& \text{if $\beta=2$}
			\end{cases}
				\quad\text{as $t\to\infty$}.
\end{equation*}
Furthermore, if  $\beta=2$, then
\begin{equation*}
	\log\sigma(t)
		= \frac{1}{2}\log\log\log t + \frac{1}{2} \log 2 + \cdots
	\quad\text{as $t\to\infty$.}
\end{equation*}
Therefore,
\begin{equation}\label{dec2}
	I_{221}(t) =
			\begin{cases}
				2^{-\ib}\frac{2\beta}{2-\beta} t^{\iib-1}\log\log t
					+ 2^{-\ib}\frac{2\beta}{2-\beta} t^{\iib-1}\log\left(\iib-1\right) + \cdots
					& \text{if $\beta\in(0,2)$}
						\\
				\sqrt{2}\log t\log\log\log t - \frac{\sqrt{2}}{2}\left(\log\log\log t\right)^2 + \cdots
					& \text{if $\beta=2$}
			\end{cases}
\end{equation}
as $t\to\infty$.
The behaviour of the term $I_{222}(t)$ can be determined as follows. Owing to
Lemma~\ref{L:Upsilon},
\begin{equation*}
	\Upsilon_{\ib-1}\left( \frac{t}{\sigma(t)}\right) =
			\begin{cases}
				\frac{2\beta}{2-\beta}\left(\frac{t}{\sigma(t)}\right)^{\iib-1}
				\log\frac{t}{\sigma(t)} - \frac{2\beta^2}{(2-\beta)^2}\left(\frac{t}{\sigma(t)}\right)^{\iib-1} +  \cdots
				& \text{if $\beta\in(0,2)$}
					\\
				\left(\log \frac{t}{\sigma(t)}\right)^2 + \cdots
				& \text{if $\beta=2$}
 			\end{cases}
		\quad\text{as $t\to\infty$.}
\end{equation*}
As a consequence, if $\beta\in(0,2)$, then
\begin{equation*}
	I_{222}(t)
		= 2^{-\ib}\frac{4\beta}{2-\beta}t^{\iib-1}\log t
			- 2^{-\ib}\frac{4\beta}{2-\beta}t^{\iib-1}\log \sigma(t)
			- 2^{-\ib}\frac{4\beta^2}{(2-\beta)^2}t^{\iib-1} + \cdots
		\quad\text{as $t\to\infty$,}
\end{equation*}
whereas, if $\beta=2$, then
\begin{equation*}
	I_{222}(t)
		= \sqrt{2}(\log t)^2-2\sqrt{2}\log t\log\sigma(t) + \cdots
			\quad\text{as $t\to\infty$.}
\end{equation*}
Inasmuch as
\begin{equation*}
	\log\sigma(t) =
			\begin{cases}
				\frac{1}{2}\log\log t + \frac{1}{2}\log\left(2\left(\iib-1\right)\right) + \cdots
					& \text{if $\beta\in(0,2)$}
						\\
				\frac{1}{2}\log\log\log t + \frac{1}{2}\log 2 + \cdots
					& \text{if $\beta=2$}
			\end{cases}
		\quad\text{as $t\to\infty$},
\end{equation*}
we conclude that
\begin{equation}\label{dec3}
	I_{222}(t) =
			\begin{cases}
				2^{-\ib}\frac{4\beta}{2-\beta}t^{\iib-1}\log t
					- 2^{-\ib}\frac{2\beta}{2-\beta} t^{\iib-1}\log \log t
					&	\\[\medskipamount]
				\quad - 2^{-\ib}\frac{2\beta}{2-\beta}t^{\iib-1}
					\left[\log\left(\iib-1\right)+\log 2+\frac{2\beta}{2-\beta}\right] + \cdots
					& \text{if $\beta\in(0,2)$}
						\\[\medskipamount]
				\sqrt{2}(\log t)^2-\sqrt{2}\log t\log\log\log t - \sqrt{2}\log 2\log t + \cdots
					& \text{if $\beta=2$}
			\end{cases}
		\quad\text{as $t\to\infty$}.
\end{equation}
Combining equations \eqref{E:I-22}, \eqref{dec2} and \eqref{dec3} tells us that
\begin{equation}\label{E:I22-final}
	I_{22}(t) =
		\begin{cases}
				2^{-\ib} \frac{4\beta}{2-\beta} t^{\iib-1}\log t
				- 2^{-\ib}\frac{2\beta}{2-\beta}
				\left(\frac{2\beta}{2-\beta} + \log 2\right) t^{\iib-1} + \cdots 				
					& \text{if $\beta\in(0,2)$}
						\\
				\sqrt{2}(\log t)^2
				- \sqrt{2}\log 2\log t + \cdots
					& \text{if $\beta=2$}
		\end{cases}
	\quad\text{as $t\to\infty$}.
\end{equation}
We finally turn our attention on the term $I_{23}(t)$.  Since
\begin{equation*}
	I_{23}(t)
	 = \int_{C\sigma(t)}^t\left(\frac{\tau^2}{2}-\frac{\sigma(t)^2}{2}\right)^{\ib-1}\d \tau
	 = 2^{1-\ib}\sigma(t)^{\iib-1}\Psi_{\ib-1}\left(\frac{t}{\sigma(t)}\right)
		 \quad\text{for $t>t_1$,}
\end{equation*}
from  Lemma~\ref{L:Psi} one can infer that
\begin{equation*}
	I_{23}(t)=2^{1-\ib}\sigma(t)^{\iib-1}
			\begin{cases}
				\frac{\beta}{2-\beta}\left(\frac{t}{\sigma(t)}\right)^{\iib-1} + \cdots
					& \text{if $\beta\in(0,2)$}
						\\
				\log\frac{t}{\sigma(t)} + \cdots
					& \text{if $\beta=2$}
			\end{cases}
				\quad\text{as $t\to\infty$.}
\end{equation*}
Hence,
\begin{equation}\label{E:I23-final}
	I_{23}(t) = \begin{cases}
					2^{-\ib} \frac{2\beta}{2-\beta} t^{\iib-1} + \cdots 				
						& \text{if $\beta\in(0,2)$}
							\\
					\sqrt{2}\log t + \cdots
						& \text{if $\beta=2$}
				\end{cases}
					\quad\text{as $t\to\infty$.}
\end{equation}
Altogether, if $\beta \in (0,2)$, then,   by
\eqref{dec15},~\eqref{E:I21-final},~\eqref{E:I22-final} and~\eqref{E:I23-final},
\begin{align} 
	\begin{split}
		I_{2}(t)
			& \le
			2^{-\ib} \frac{\beta}{2+\beta} t^{\iib+1}
				- 2^{-\ib}t^{\iib-1}\log t\left(\iib-\frac{1-\beta}{\beta^2}\frac{4\beta}{2-\beta}\right)
						\\
			& \quad + 2^{-\ib}\frac{2}{2-\beta}t^{\iib-1}
				\left[\log c_{\beta}-\frac{1-\beta}{\beta^2}\left(\frac{2\beta}{2-\beta}+\log 2\right) + \beta K^{+}_{\varepsilon}\right] +\cdots 		
					\quad\text{as $t\to\infty$,}
	\end{split}
\end{align}
and, by making use of \eqref{dec16} instead of \eqref{dec15},
\begin{align} 
	\begin{split}
		I_{2}(t)
			& \ge  2^{-\ib} \frac{\beta}{2+\beta} t^{\iib+1}
				- 2^{-\ib}t^{\iib-1}\log t\left(\iib-\frac{1-\beta}{\beta^2}\frac{4\beta}{2-\beta}\right)
						\\
			& \quad + 2^{-\ib}\frac{2}{2-\beta}t^{\iib-1}
				\left[\log c_{\beta}-\frac{1-\beta}{\beta^2}\left(\frac{2\beta}{2-\beta}+\log 2\right)  + \beta K^{-}_{\varepsilon}\right] +\cdots
					\quad\text{as $t\to\infty$}.
	\end{split}
\end{align}
Therefore, owing to the arbitrariness  of $\varepsilon$,
\begin{align}\label{E:summary-for-I2-total}
	\begin{split}
		I_{2}(t)
			& = 2^{-\ib} \frac{\beta}{2+\beta} t^{\iib+1}
				- 2^{-\ib}\frac{2}{2-\beta}t^{\iib-1}\log t
				+ 2^{-\ib} c_{\beta,N}t^{\iib-1} + \cdots
		        			\quad\text{as $t\to\infty$}
	\end{split}
\end{align}
for a suitable constant $c_{\beta,N} \in \R$.

Similarly, if $\beta=2$, then
\begin{align}\label{E:summary-for-I2-upper-beta-2}
\begin{split}
	I_{2}(t)
		& = \frac{1}{2\sqrt{2}} t^2
				- \frac{1}{\sqrt{2}} \log t\log\log t
				+ \frac{1}{\sqrt{2}} \log c_2\log t
				+ \cdots
					\\
			& \quad - \frac{1}{4}\left(\sqrt{2}(\log t)^2
				- \sqrt{2}\log 2\log t + \cdots\right)
				- \frac{\log(2N)}{2}\sqrt{2}\log t
				+ \cdots
					\\
		& = \frac{1}{2\sqrt{2}}t^2-\frac{\sqrt{2}}{4}(\log t)^2
				- \frac{1}{\sqrt{2}}\log t\log\log t
				+ \cdots
	\quad\text{as $t\to\infty$}.
\end{split}
\end{align}
Owing to equation ~\eqref{E:estimate-of-I1},  formulas
\eqref{E:summary-for-I2-total} and \eqref{E:summary-for-I2-upper-beta-2}
continue to hold with $I_2(t)$ replaced by $I_1(t) + I_2(t)$ on the left-hand
side. Hence, if $\beta\in(0,2)$, then by~\eqref{E:splitting-of-Is},
\begin{equation*}
\normIB = 2^{-\ib} \frac{\beta}{2+\beta} t^{\iib+1}
				- 2^{-\ib}\frac{2}{2-\beta}t^{\iib-1}\log t
				+ 2^{-\ib} c_{\beta,N}t^{\iib-1} + \cdots
		        			\quad\text{as $t\to\infty$,}
\end{equation*}
and, if $\beta=2$, then
\begin{equation*}
	\normIB = \frac{1}{2\sqrt{2}}t^2 - \frac{\sqrt{2}}{4}(\log t)^2 - \frac{1}{\sqrt{2}}\log t\log\log t+\cdots
		\quad\text{as $t\to\infty$.}
\end{equation*}
The proof is complete.
\end{proof}

Lemma~\ref{L:I-norm-sharp}, coupled with equation \eqref{E:power2}, immediately
implies the following result.

\begin{lemma} \label{L:F-B-expansion}
Let $B$ be a Young function of the form \eqref{BN} for some $\beta \in (0,2]$
and $N>0$. Let $\FF_B$ be the function defined by \eqref{E:F-LB}, and let
$\kappab$ be the constant given by  \eqref{E:kappab}.  Then
\begin{equation*}
	\bigl[ \kappab \FF_\beta(t) \bigr]^\frac{2\beta}{2+\beta}
		= \frac{t^2}{2} +
		\begin{cases}
			- \frac{2}{2-\beta}\log t
			+ c_{\beta,N} + \cdots
				& \text{if $\beta\in(0,2)$}
				\\
			- \frac{1}{2}(\log t)^2
			- \log t \log\log t
			+ \cdots
				& \text{if $\beta=2$}
		\end{cases}
		\quad\text{as $t\to\infty$},
\end{equation*}
where $c_{\beta,N}$ is the constant appearing in Lemma~\ref{L:I-norm-sharp}.
\end{lemma}

We are now in a position to accomplish the proof of
Theorem~\ref{T:integral-form-improved}.

\begin{proof}[Proof of Theorem~\ref{T:integral-form-improved}]
Assume that $u$ is a weakly differentiable function  satisfying conditions
\eqref{E:nabla-integral} and \eqref{E:mu}. Then, by
Lemma~\ref{L:Holder-estimates}, there exists a Young function $B$ of the form
\eqref{BN} such that
\begin{equation} \label{may8}
	\int_{\rn}
		\exp^{\frac{2\beta}{2+\beta}}\left(\kappab |u| \right)
		\EF\left(|u|\right)\dgn
		\le \sqrt{\frac{2}{\pi}}
			\int_{0}^{\infty} e^{\left[ \kappab\FF_B(t) \right]^{\frac{2\beta}{2+\beta}}
				-\frac{t^2}{2}}
			\EF\left( \FF_B(t) \right)
				\d t\,.
\end{equation}
Thanks to Lemma~\ref{L:F-B-expansion},
\begin{equation} \label{E:F-B-power}
	\bigl[ \kappab \FF_B(t) \bigr]^\frac{2\beta}{2+\beta} - \frac{t^2}{2}
		= \begin{cases}
			- \frac{2}{2-\beta}\log t + c_{\beta,N} + \cdots
				& \text{if $\beta\in(0,2)$}
				\\
			- \frac{1}{2}(\log t)^2 -\log t\log\log t + \cdots
				& \text{if $\beta=2$}
		\end{cases}
		\quad\text{as $t\to\infty$.}
\end{equation}
Moreover,  from Lemma~\ref{L:I-norm-sharp} and the
definition of $\FF_B$, one can deduce that
\begin{equation} \label{E:F-B-upper}
	\FF_B(t)
		\le \mu_\beta\, t^{\iib+1} + \cdots
		\quad\text{near infinity,}
\end{equation}
where $\mu_\beta=2^{-\ib}\frac{\beta}{2+\beta}$.

Assume first that $\beta\in(0,2)$. The integral on the right hand side of \eqref{may8}
converges
if
\begin{equation} \label{may9}
	\int^\infty e^{-\frac{2}{2-\beta}\log t}\EF\left( \FF_B(t) \right)\d t
		< \infty.
\end{equation}
Since $\EF$ is an increasing function, equation \eqref{E:F-B-upper} implies that
\begin{equation*} 
	\EF\left( \FF_B(t) \right)
		\le \EF\left( \mu_\beta\, t^{\iib+1} \right)
		\quad\text{for large $t$.}
\end{equation*}
Therefore,  equation \eqref{may9} holds provided that
\begin{equation*}
	\int^\infty e^{-\frac{\beta}{2-\beta}\log t}
		\EF\left( \mu_\beta\, t^{\iib+1} \right)\d t
		< \infty.
\end{equation*}
The convergence of the last integral follows from
assumption \eqref{E:improvement-conditions-beta-less-than-2}.

Assume next that $\beta=2$.
Owing  to equation \eqref{E:F-B-power}, given any $\delta>0$, one has that
\begin{equation*}
	\bigl[ \kappab \FF_B(t) \bigr]^\frac{2\beta}{2+\beta} - \frac{t^2}{2}
		\le - \frac{1}{2}(\log t)^2 + (\delta-1)\log t\log\log t
		\quad\text{for large $t$.}
\end{equation*}
Hence, in the light of inequality \eqref{E:F-B-upper}, the integral on
the right-hand side of \eqref{may8} converges if there exists $\delta >0$ such that
\begin{equation*}
	\int^{\infty}e^{-\frac{1}{2}(\log t)^2 + (\delta-1)\log t\log\log t}
		\EF\left(\mu_2\, t^2\right)\,\d t <\infty.
\end{equation*}
The last inequality follows from
assumption~\eqref{E:improvement-conditions-beta-2}, via the change of variables
$s=\mu_2\, t^2$.  The validity of inequality \eqref{E:sup-improved} is thus
established  for both $\beta\in(0,2)$ and $\beta =2$, under the respective
assumptions.

It remains to
prove the sharpness of these assumptions.
Let $\tau_0>0$ and let $f\colon(\tau_0,\infty)\to(0,\infty)$
be a locally integrable function. Define the function $u\colon\rn\to\R$ as
\begin{equation*}
	u(x) = \sgn x_1
	\begin{cases}
		\displaystyle
			0
			& \text{for $|x_1|<\tau_0$}
			\\
		\displaystyle
		\int_{\tau_0}^{|x_1|} f(\tau)\,\d \tau
			& \text{for $|x_1| \ge  \tau_0$}.
	\end{cases}
\end{equation*}
Then $u$ is weakly differentiable, $\med(u)=\mv(u)=0$
and
\begin{equation*}
	|\nabla u(x)|
	 =
	\begin{cases}
		0
			& \text{for $|x_1|<\tau_0$}
			\\
		f\bigl( |x_1| \bigr)
			& \text{for \ae $x$ such that $|x_1|>\tau_0$}.
	\end{cases}
\end{equation*}
Fix $t_0>0$ so large that  $\Exp^\beta(t)=e^{t^\beta}$ for $t>t_0$, and then
choose $\tau_0>0$ such that
\begin{equation*}
	\frac{e^{\frac{\tau^2}{2}}}{\tau(\log\tau)^{2}}
		\ge \Exp^{\beta}(t_0)
		\quad\text{for $\tau>\tau_0$}.
\end{equation*}
Set
\begin{equation} \label{E:f-def}
	f(\tau) = E^{-1}\left( \frac{e^{\frac{\tau^2}{2}}}{\tau(\log \tau)^{2}} \right)
		\quad\text{for $\tau>\tau_0$},
\end{equation}
where $E^{-1}$ denotes the inverse of the function $\Exp^\beta$ on $(t_0,\infty)$.
Thus,
\begin{equation*}
	\int_{\rn} \Exp^\beta(|\nabla u|)\dgn
		= 2\int_{0}^{\tau_0} \d\gamma_1
		+ \frac{2}{\sqrt{2\pi}} \int_{\tau_0}^{\infty}
			\Exp^\beta\bigl(f(\tau)\bigr) e^{-\frac{\tau^2}{2}}\,\d\tau
		\le 1 + \frac{2}{\sqrt{2\pi}} \int_{\tau_0}^{\infty}
			\frac{\d\tau}{\tau(\log\tau)^{2}}.
\end{equation*}
Since $M>1$ and the last integral converges, we may assume, on increasing $\tau_0$, if necessary, that
\begin{equation*}
	\int_{\rn} \Exp^\beta(|\nabla u|)\,\dgn \le M\,.
\end{equation*}
As far as the integral in \eqref{E:integral-improved-sharp} is concerned,
we have that
\begin{align} \label{E:integral-improved-f}
	\begin{split}
	& \int_{\rn} \exp^\frac{2\beta}{2+\beta}(\kappab|u|)
		\EF(|u|)\,\dgn
			\\
		&\qquad \ge \frac{2}{\sqrt{2\pi}} \int_{\tau_0}^{\infty}
			\exp\left\{
				\left( \kappab
					\int_{\tau_0}^{t} f(\tau)\,\d\tau
				\right)^\frac{2\beta}{2+\beta}
				- \frac{t^2}{2}
			\right\}
				\EF\left(
					\int_{\tau_0}^{t} f(\tau)\,\d\tau
					\right) \d t.
	\end{split}
\end{align}
Since $E^{-1}(\tau)=(\log \tau)^\ib$ if $\tau \geq \Exp^{\beta}(t_0)$, the
function $f$ defined by equation \eqref{E:f-def} takes the form
\begin{equation*}
	f(\tau)
		= \left( \log\frac{e^{\frac{\tau^2}{2}}}{\tau(\log \tau)^{2}} \right)^\ib
		= \left( \frac{\tau^2}{2} - \log\tau - 2\log\log\tau \right)^\ib
		\quad\text{for $\tau>\tau_0$.}
\end{equation*}
Routine computations resting upon L'H\^opital's rule then tell us that
\begin{equation*}
	\int_{\tau_0}^{t} f(\tau) \,\d \tau
		 =  2^{-\ib}\frac{\beta}{2+\beta} t^{\iib+1}
				+ 2^{-\ib}
				\begin{cases}
					- \frac{2}{2-\beta}t^{\iib-1}\log t
					- \frac{4}{2-\beta} t^{\ib-1} \log\log t
					+ \cdots
						& \text{if $\beta\in(0,2)$}
						\\
					- \frac{1}{2}(\log t)^2
					- 2\log t\log\log t
					+ \cdots
						& \text{if $\beta=2$}
				\end{cases}
\end{equation*}
as $t\to\infty$.
Consequently, by formula \eqref{E:power},
\begin{equation*}
	\left( \kappab
		\int_{\tau_0}^{t} f(\tau) \,\d \tau
	\right)^\frac{2\beta}{2+\beta}
		- \frac{t^2}{2} =
		\begin{cases}
			- \frac{2}{2-\beta}\log t
			- \frac{4}{2-\beta}\log\log t
			+ \cdots
				& \text{if $\beta\in(0,2)$}
				\\
			- \frac{1}{2}(\log t)^2
			- 2 \log t\log\log t
			+ \cdots
				& \text{if $\beta=2$}
		\end{cases}
		\quad\text{as $t\to\infty$}.
\end{equation*}
Now, there exists $\tau_1>\tau_0$ such that
\begin{equation} \label{sep10}
	\int_{\tau_0}^{t} f(\tau) \,\d \tau
		 \ge  \frac{\mu_\beta}{2} t^{\frac{2}{\beta}+1}
		\quad\text{for $t>\tau_1$},
\end{equation}
where we have set $\mu_\beta = 2^{-\ib} \frac{\beta}{2+\beta}$,
and, simultaneously,
\begin{equation} \label{sep11}
	\left( \kappab
		\int_{\tau_0}^{t} f(\tau) \,\d \tau
	\right)^\frac{2\beta}{2+\beta}
		- \frac{t^2}{2}
		\ge - g(t)
		\quad\text{for $t>\tau_1$},
\end{equation}
where we have set
\begin{equation*}
	g(t) =
		\begin{cases}
			\frac{2}{2-\beta}\log t
			+ \frac{8}{2-\beta}\log\log t 	
				& \text{if $\beta\in(0,2)$}
				\\
			\frac{1}{2}(\log t)^2
			+ 4 \log t\log\log t
				& \text{if $\beta=2$}
		\end{cases}
\end{equation*}
for large $t$.
Therefore, since $\EF$ is an increasing function, from inequality
\eqref{E:integral-improved-f} we can deduce that
\begin{equation} \label{E:integral-improved-fb}
	\int_{\rn} \exp^\frac{2\beta}{2+\beta}(\kappab|u|)
		\EF(|u|)\,\dgn
		\ge \frac{2}{\sqrt{2\pi}}
		\int_{\tau_1}^{\infty}
			e^{-g(t)}
				\EF\left(
					\frac{\mu_\beta}{2} t^{\iib+1}
					\right)\d t\,.
\end{equation}
Observe that, in deriving inequality \eqref{E:integral-improved-fb}, we have
exploited the lower bounds \eqref{sep10}, \eqref{sep11} and replaced the lower
limit of integration $\tau_0$ by $\tau_1$.

First, assume that $\beta\in(0,2)$. The integral on the right-hand side of
\eqref{E:integral-improved-fb} diverges if
\begin{equation}\label{dec40a}
	\int^\infty t^{-\frac{2}{2-\beta}}
		(\log t)^{-\frac{8}{2-\beta}}
		\EF\left( \frac{\mu_\beta}{2} t^{\iib+1} \right)\d t
		= \infty\,.
\end{equation}
By a change of variables, equation \eqref{dec40a} is  equivalent to
\begin{equation*}
	\int^\infty t^{-\frac{4}{4-\beta^2}}
		(\log t)^{-\frac{8}{2-\beta}}
		\EF(t)\,\d t
		= \infty.
\end{equation*}
Thus, equation \eqref{E:integral-improved-sharp} follows via
\eqref{E:integral-improved-fb}, by assumption
\eqref{E:improvement-sharpness-beta-less-than-2}.

Next, suppose that $\beta=2$. The integral on the right-hand side of
\eqref{E:integral-improved-fb} diverges if
\begin{equation}\label{dec41}
	\int^\infty
		e^{-\frac{1}{2}(\log t)^2 - 4\log t\log\log t}
		\EF\left( \frac{\mu_2}{2} t^2 \right)\d t
		= \infty\,.
\end{equation}
A change of variables again shows that equation \eqref{dec41} certainly holds
provided that
\begin{equation*}
	\int^\infty
		e^{-\frac{1}{8}(\log t)^2 - \frac{5}{2}\log t\log\log t}
		\EF(t)\,\d t
		= \infty.
\end{equation*}
Equation \eqref{E:integral-improved-sharp} now follows from a combination of
\eqref{E:integral-improved-fb} and \eqref{E:improvement-sharpness-beta-2}.
\end{proof}

\section{Proof of Theorem~\ref{T:maximizers-integral}}

Besides Theorem~\ref{T:integral-form-improved}, some classical results of
functional analysis come into play in our proof of
Theorem~\ref{T:maximizers-integral}. They are recalled below, in a form
suitable for our applications.

\begin{theoremalph}[Equi-integrability (de la Vall\'ee-Poussin)]
\label{TA:delaVallePoussin}
Let $(\RR,\nu)$ be a probability space. Then a sequence
$\{u_k\}\subset L^1(\RR,\nu)$ is equi-integrable
if and only if there exists a convex function $\Psi\colon [0, \infty) \to [0, \infty)$, satisfying $\lim_{t \to \infty} \Psi(t)/t=\infty$, such that
\begin{equation*} 
	\sup_{k} \int_{\RR} \Psi(|u_k|)\,\d\nu < \infty\,.
\end{equation*}
\end{theoremalph}

\begin{theoremalph}[Convergence in $L^1$ (Vitali)]
\label{TA:Vitali}
Let $(\RR,\nu)$ be a probability space and let
$\{u_k\}$  be a sequence in $L^1(\RR,\nu)$. Then $\{u_k\}$  converges to a function  $u \in L^1(\RR,\nu)$ if
and only if $\{u_k\}$ is equi-integrable and converges  to $u$  in measure.
\end{theoremalph}

The following result is a consequence of
\citep[Section~1.2, Theorem~8]{Gia:98a}.

\begin{theoremalph}[Semicontinuity (Serrin)]
\label{TA:Serrin}
Let $g\colon\rn\to [0,\infty)$ be a convex function.
Then the functional defined as
\begin{equation*}
	\int_{\rn} g(\nabla u) \,\dgn
\end{equation*}
for a weakly differentiable function $u\colon\rn\to\R$ is sequentially lower
semicontinuous with respect to weak$^{*}$ convergence in $W^{1,1}\RG$.
\end{theoremalph}

An application of the  Banach Alaoglu theorem to the space $\expLb$ yields the
next theorem.  Note that the space $\expLb$ is the dual of the separable space
$L(\log L)^\ib\RG$.

\begin{theoremalph}[Weak$^{*}$ compactness in $\expLb$ (Banach--Alaoglu)]
\label{TA:Banach-Alaoglu}
Assume that $\{u_k\}$
is a bounded sequence in $\expLb$.
Then there
exist $u\in\expLb$ and
a subsequence $\{u_{k_\ell}\}$ such that
$u_{k_\ell}\to u$ in the weak$^*$ topology of $\expLb$.
\end{theoremalph}

The last preliminary result, concerning convergence of medians, is contained in
Lemma~\ref{L:mu-limit}.

\begin{lemma} \label{L:mu-limit}
Assume that $u \in W^{1,1}\RG$ and that
$\{u_k\}$ is a sequence in $W^{1,1}\RG$
such that $u_k\to u$ in $L^1\RG$.
Then there is a subsequence of $\{u_k\}$, still denoted by $\{u_k\}$, such that
\begin{equation} \label{E:lim-mu}
	\lim_{k\to\infty} \m(u_k) = \m(u).
\end{equation}
Here, $m(u)$ denotes either the mean value or the median of $u$.
\end{lemma}

\begin{proof}
The claim is trivial if $\m(\cdot)$ is the mean value.  Let us consider the
case when  $\m(\cdot)$ is the median.  As recalled in Proposition~\ref{P:PSintegral},
the functions $u_k^\circ$, for $k\in\N$, and $u^\circ$ are
(locally absolutely) continuous, since $u_k$ and $u$ are Sobolev functions.  As
a consequence of \cite[Chapter~3, Theorem~7.4]{Ben:88}, the signed decreasing
rearrangement is a contraction from $L^1\RG$ in $L^1(0,1)$. Therefore, since
$u_k\to u$ in $L^1\RG$, one  also has that $u_k^\circ \to u^\circ$ in
$L^1(0,1)$. Thus, there exists a subsequence of $\{u_k\}$, still denoted by
$\{u_k\}$, such that
\begin{equation} \label{E:uk-circ-convergence}
	\lim_{k\to\infty}  u_k^\circ = u^\circ
		\quad\text{\ae in $(0,1)$.}
\end{equation}
We claim that \eqref{E:uk-circ-convergence} holds, in fact, everywhere in  $(0,1)$. To verify this claim,
assume, by contradiction, that \eqref{E:uk-circ-convergence}
is violated for some $s_0\in(0,1)$, namely
there exists $\varepsilon>0$ such that
\begin{equation*}
	|u_{k_\ell}^\circ(s_0) - u^\circ(s_0)| > \varepsilon
\end{equation*}
for some subsequence $\{u^\circ_{k_\ell}\}$. Hence, either the set
$U=\{\ell: u_{k_\ell}^\circ(s_0) > u^\circ(s_0)+\varepsilon\}$
or the set
$L=\{\ell: u_{k_\ell}^\circ(s_0) < u^\circ(s_0)-\varepsilon\}$
is infinite. Assume, for instance, that $U$ is infinite, the proof in
the case when $L$ is infinite being analogous. By the continuity of
$u^\circ$, there exists $\delta>0$ such that
$|u^\circ(s)-u^\circ(s_0)|<\varepsilon/2$ for every $s\in(0,1)$ obeying
$|s-s_0|<\delta$. Since  equation \eqref{E:uk-circ-convergence} holds for
\ae $s\in(0,1)$, there exists $r_0\in(s_0-\delta,s_0)$ such that
$u_{k_\ell}^\circ(r_0) \to u^\circ(r_0)$ as $\ell\to\infty$. On the other hand,
\begin{equation*}
	u^\circ(r_0) + \frac{\varepsilon}{2}
		< u^\circ(s_0) + \varepsilon
		< u_{k_\ell}^\circ(s_0)
		\le u_{k_\ell}^\circ(r_0)
		\quad\text{for $\ell\in U$}
\end{equation*}
since  the functions $u_{k_\ell}^\circ$ are non-increasing.  This contradicts
the convergence of the sequence $\{u_{k_\ell}^\circ(r_0)\}$ to $u^\circ(r_0)$.
Equation \eqref{E:lim-mu}  follows from our claim, since
\begin{equation*}
	\lim_{k\to\infty} \med(u_k)
		= \lim_{k\to\infty} u_k^\circ (\tfrac 12)
		= u^\circ (\tfrac 12)= \med(u).
		\qedhere
\end{equation*}
\end{proof}

We are now ready to prove Theorem~\ref{T:maximizers-integral}.

\begin{proof}[Proof of Theorem~\ref{T:maximizers-integral}]
Let $\{u_k\}\subset\WexpLb$ be a maximizing
sequence  in \eqref{E:sup-integral}, namely $\m(u_k)=0$ and
\begin{equation} \label{sep5}
	\int_{\rn} \expb(|\nabla u_k|)\,\dgn
		\le M
\end{equation}
for $k\in\N$, and
\begin{equation*}
	\lim_{k\to\infty} \int_{\rn}
		\exp^\frac{2\beta}{2+\beta}\left(\kappab |u_k|\right)\dgn
		= S,
\end{equation*}
where $S$ denotes   the supremum in \eqref{E:sup-integral}.

Since $\WexpLb$ is compactly embedded into $L^1\RG$ (see \eg
\citep[Theorem~7.3]{Sla:15}), there exists a  subsequence of $\{u_k\}$, still
denoted by $\{u_k\}$, such that  $u_k\to u$ in $L^1\RG$ and
\begin{equation} \label{E:uk-converges-ae}
	u_k\to u
		\quad\text{\ae in $\rn$}.
\end{equation}
Also, by Lemma~\ref{L:mu-limit}, we may assume that  $\m(u_k)\to\m(u)$,
whence we infer that $\m(u)=0$.

By inequality \eqref{sep5} and Theorem~\ref{TA:Banach-Alaoglu}, there exists  a
measurable function $V\colon\rn\to\rn$, with $|V|\in\exp L^\beta\RG$, and
a subsequence of $\{\nabla u_k\}$, still denoted by $\{\nabla u_k\}$, such
that $\nabla u_k\to V$ in the weak$^*$ topology of $\exp L^\beta\RG$.  By the
definition of weak gradient, we have that $u$ is weakly differentiable, $\nabla
u=V$, $u\in\WexpLb$, and $\nabla u_k\to\nabla u$ in the weak$^*$ topology
of $\exp L^\beta\RG$. In particular, $\nabla u_k\to\nabla u$ in the weak$^*$
topology of $L^1\RG$. Consequently, owing to Theorem~\ref{TA:Serrin} and
inequality \eqref{sep5},
\begin{equation*}
	\int_{\rn} \expb(|\nabla u|)\,\dgn
		\le\liminf_{k\to\infty} \int_{\rn} \expb(|\nabla u_k|)\,\dgn
		\le M.
\end{equation*}

Our next task is to show that
\begin{equation} \label{E:u-attains-sup-int}
	\int_{\rn} \exp^\frac{2\beta}{2+\beta}\left( \kappab|u| \right)\,\dgn
		= S.
\end{equation}
Thanks to Theorem~\ref{T:integral-form-improved}, there exist a continuously
increasing function $\EF\colon[0,\infty)\to[0,\infty)$  satisfying
$\lim_{t\to\infty}\EF(t)=\infty$ and a constant $C$ such that
\begin{equation} \label{sep3}
	\int_{\rn} \exp^\frac{2\beta}{2+\beta}\left( \kappab|u_k| \right)
		\EF\left( |u_k| \right)\dgn
		\le C
\end{equation}
for  $k\in\N$. Define the function $\Theta\colon [0, \infty) \to [0, \infty)$ as
\begin{equation*}
	\Theta(t)
		= t \,\EF\left( \frac{1}{\kappab}(\log t)^{\frac{1}{2}+\ib}\right)
		\quad\text{for $t\ge 1$}
\end{equation*}
and $\Theta(t)=0$ for $t\in[0,1)$.
Then the function $\Theta(t)/t$ is non-decreasing
and $\lim _{t \to \infty }\Theta(t)/t= \infty$.
Finally, define the function $\Psi\colon[0,\infty)\to[0,\infty)$ by
\begin{equation*}
	\Psi(t) = \int_{0}^{t} \frac{\Theta(\tau)}{\tau}\,\d\tau
		\quad\text{for $t\ge 0$.}
\end{equation*}
Then $\Psi$ is a Young function such that
$\lim _{t\to \infty}\Psi(t)/t=\infty$ and $\Psi(t)\le\Theta(t)$
for $t\in[0,\infty)$.
Thus, owing to \eqref{sep3},
\begin{align*}
	\sup_{k\in\N} \int_{\rn} \Psi\left(
		\exp^\frac{2\beta}{2+\beta}\left( \kappab|u_k| \right)\right)\dgn
	& \le \sup_{k\in\N} \int_{\rn} \Theta \left(
		\exp^\frac{2\beta}{2+\beta}\left( \kappab|u_k| \right)\right)\dgn
		\\
	& = \sup_{k\in\N}
		\int_{\rn} \exp^\frac{2\beta}{2+\beta}\left( \kappab|u_k| \right)
		\EF\left( |u_k| \right)\dgn
		\le C\,.
\end{align*}
As a consequence of Theorem~\ref{TA:delaVallePoussin}, the sequence of functions
$\left\{\exp^\frac{2\beta}{2+\beta}\left( \kappab|u_k| \right)\right\}$ is equi-integrable
in $L^1\RG$.  Also, by \eqref{E:uk-converges-ae}, this sequence
converges to $\exp^\frac{2\beta}{2+\beta}(\kappab|u|)$ \ae in $\rn$, and
hence it converges in measure to the same function. From Theorem~\ref{TA:Vitali},
we deduce that
\begin{equation*}
	\lim_{k\to\infty}
		\int_{\rn} \exp^\frac{2\beta}{2+\beta}\left( \kappab|u_k|\right)\dgn
		= \int_{\rn} \exp^\frac{2\beta}{2+\beta}\left( \kappab|u|\right)\dgn\,,
\end{equation*}
whence equation \eqref{E:u-attains-sup-int} follows.
This shows that $u$ is actually a maximizer for \eqref{E:sup-integral}.

It remains to show that if  $u$ is a maximizer, then it has necessarily  the
form \eqref{sep220}.  Assume, by contradiction, that this is not the case. Then
$u\ne u^\bullet$ (even up to rotations about $0$). Hence, owing to 
Proposition~\ref{P:PSintegral}, applied with $A(t)= \Exp^\beta (t) - \Exp^\beta(0)$,
\begin{equation} \label{oct1}
	\int_{\rn} \Exp^\beta\left(|\nabla u^\bullet|\right) \dgn
		< \int_{\rn} \Exp^\beta\left(|\nabla u|\right) \dgn
		\leq M.
\end{equation}
Moreover,
\begin{equation} \label{oct2}
	\int_{\rn} \exp^{\frac{2\beta}{2+\beta}}\left(\kappab |u^\bullet| \right)\dgn
		= \int_{\rn} \exp^{\frac{2\beta}{2+\beta}}\left(\kappab |u| \right)\dgn.
\end{equation}
Therefore,  the supremum in  \eqref{E:sup-integral} is also attained at
$u^\bullet$.  We have that
\begin{equation*}
	u^\bullet (x) = \phi (x_1)
		\quad\text{for $x\in\rn$},
\end{equation*}
for some non-decreasing locally absolutely continuous function  $\phi\colon\R\to\R$.
Note that
\begin{equation} \label{oct3}
	\int_{\rn} \Exp^\beta\left(|\nabla u^\bullet|\right) \dgn
		= \int_{\R} \Exp^\beta\left(\phi'\right)\d\gamma_1
\end{equation}
and
\begin{equation} \label{oct4}
	\int_{\rn} \exp^{\frac{2\beta}{2+\beta}}\left(\kappab |u^\bullet| \right)\dgn
		= \int_{\R} \exp^{\frac{2\beta}{2+\beta}}\left(\kappab |\phi| \right)\d\gamma_1.
\end{equation}
By equations \eqref{oct1} and \eqref{oct3},
\begin{equation} \label{oct5}
	\int_{\R} \Exp^\beta\left(\phi'\right) \d\gamma_1
		< M
\end{equation}
and, thanks to \eqref{oct2} and \eqref{oct4},
\begin{equation} \label{oct12}
	\int_{\R} \exp^{\frac{2\beta}{2+\beta}}\left(\kappab |\phi| \right)\d\gamma_1
		= \int_{\rn} \exp^{\frac{2\beta}{2+\beta}}\left(\kappab |u| \right)\dgn.
\end{equation}
Define, for $\lambda >1$, the function $\eta_\lambda \colon\R\to [0,\infty)$ as
\begin{equation} \label{oct6}
	\eta_\lambda (t) = \big(\phi '(t)^\beta + \log\lambda \big)^\ib
		\quad\text{for $t\in\R$}.
\end{equation}
Assume first that $\m(u)$ stands for  $\med(u)$ in \eqref{E:sup-integral}.
Then  $\med(\phi)=  \med(u)=0$, and hence $\phi (0)=0$. Also, $\phi(t) \ge 0$
if $t \ge 0$  and   $\phi(t) \le 0$  if $t \leq 0$.  Define, for $\lambda >0$,
the function $\psi_\lambda \colon\R\to\R$ as
\begin{equation*}
	\psi_\lambda (t) = \int_{0}^t \eta_\lambda (\tau) \,\d\tau
		\quad\text{for $t\in\R$},
\end{equation*}
and the function $v_\lambda \colon\rn\to\R$ as
\begin{equation} \label{oct11-a}
	v_\lambda (x) = \psi_\lambda (x_1)
		\quad\text{for $x\in\rn$}.
\end{equation}
One has that $\med(v_\lambda)= \psi_\lambda(0)=0$. Next, let $t_0\geq 0$ be such that
\begin{equation*}
	\Exp^\beta (t) =
		\begin{cases}
			\frac{t}{t_0} \left(e^{t_0^\beta} -1\right) + 1
				& \text{for $t\in [0, t_0)$}
					\\
			e^{t^\beta}
				& \text{for $t\in [t_0, \infty)$.}
		\end{cases}
\end{equation*}
Of course, $t_0=0$ if $\beta \geq 1$. Moreover,
\begin{align}\label{dec21}
\begin{split}
	\lim_{\lambda\to1^+}\int_{\rn} \Exp^{\beta}(|\nabla v_\lambda|) \,\dgn
		& = \lim_{\lambda\to 1^+} \int_{\R} \Exp^{\beta}(\eta_\lambda)\,\d\gamma_1
			\\
		& = \lim_{\lambda\to 1^+}		
			\bigg(\frac{e^{t_0^\beta}-1}{t_0}
				\int_{\{\eta_\lambda < t_0\}} \eta_\lambda\,\d\gamma_1
				+ \int_{\{\eta_\lambda < t_0\}} \d\gamma_1
				+ \int_{\{\eta_\lambda\ge t_0\}} e^{\eta_\lambda^\beta} \,\d\gamma_1
			\bigg)
			\\
		& = \frac{e^{t_0^\beta}-1}{t_0}\int_{\{\phi' < t_0\}} \phi' \,\d\gamma_1
				+ \int_{\{\phi' < t_0\}} \d\gamma_1
				+ \int_{\{\phi' \ge t_0\}} e^{(\phi')^\beta} \,\d\gamma_1
			\\
		& = \int_{\rn} \Exp^{\beta}(\phi') \,\d\gamma_1\,,
\end{split}
\end{align}
where the third equality holds thanks to the dominated convergence theorem.
By equations \eqref{oct5} and \eqref{dec21},
\begin{equation} \label{oct8}
	\int_{\rn} \Exp^\beta(|\nabla v_\lambda|) \,\dgn
		< M
\end{equation}
provided that $\lambda$ is sufficiently close to $1$.
On the other hand, since $\eta_\lambda (t)>\phi'(t)\ge 0$ in $\R$, one has that
$|\psi_\lambda (t)|>|\phi(t)|$ for $t\in\R\setminus\{0\}$, and hence
\begin{equation} \label{oct9}
	\int_{\rn} e^{\left(\kappab |v_\lambda| \right)^{\frac{2\beta}{2+\beta}}} \,\dgn
		= \int_{\R} e^{\left(\kappab |\psi_\lambda| \right)^{\frac{2\beta}{2+\beta}}} \,\d\gamma_1
		> \int_{\R} e^{\left(\kappab |\phi| \right)^{\frac{2\beta}{2+\beta}}} \,\d\gamma_1
		= \int_{\rn} e^{\left(\kappab |u| \right)^{\frac{2\beta}{2+\beta}}} \,\dgn
\end{equation}
thanks to \eqref{oct12}.
Altogether, the maximizing property of $u$ is contradicted.

Next, suppose  that $m(u)$ stands for  $\mv(u)$ in \eqref{E:sup-integral}.
Since $\phi$ is continuous, there exists  $t_0 \in \R$ such that $\phi (t_0)
=\mv(u) =0$. Let $\eta_\lambda$ be as in \eqref{oct6}, let $\psi_\lambda$ be
defined by
\begin{equation*}
	\psi_\lambda (t) = \int_{t_0}^t \eta_\lambda (\tau) \,\d\tau
		\quad\text{for $t\in\R$,}
\end{equation*}
and  let $v_\lambda$ be the function associated with $\psi_\lambda$ as in
\eqref{oct11-a}. Since equation \eqref{dec21} continues to hold, one can choose
$\lambda>1$ such that equation \eqref{oct8} is fulfilled.  If
$\mv(\psi_\lambda)=0$, then we obtain a contradiction as above. Assume that,
instead, $\mv(\psi_\lambda )\ne 0$, say $\mv(\psi_\lambda)> 0$, to fix ideas.
Therefore,
\begin{equation*}
	\int_{t_0}^\infty \psi_\lambda  \,\d \gamma_1
		> - \int_{-\infty}^{t_0} \ \psi_\lambda\,\d\gamma_1.
\end{equation*}
Consequently, there exists $\theta \in (1,\lambda)$ such that, on defining
$\eta_{\lambda, \theta}\colon\R\to [0,\infty)$ as
\begin{equation*}
 \eta_{\lambda, \theta} (t) =
		\begin{cases}
			\big(\phi '(t)^\beta + \log \theta \big)^\ib
				& \text{if $t\geq t_0$}
					\\
			\big(\phi '(t)^\beta + \log \lambda \big)^\ib
				& \text{if $t<t_0$},
		\end{cases}
\end{equation*}
and $ \psi_{\lambda, \theta}\colon\R\to\R$ as
\begin{equation*}
	\psi_{\lambda, \theta}(t)
		=\int_{t_0}^t \eta_{\lambda, \theta}(\tau) \,\d\tau
			\quad\text{for $t\in\R$},
\end{equation*}
one has that $\mv(\psi_{\lambda, \theta})=0$. If the function $v_{\lambda, \theta}$
is defined as in \eqref{oct11-a}, with $\psi_\lambda$ replaced by
$\psi_{\lambda, \theta}$, then equations \eqref{oct8} and \eqref{oct9} still hold,
with $v_\lambda$ and $\psi_\lambda$ replaced by $v_{\lambda, \theta}$ and
$\psi_{\lambda, \theta}$. The maximizing property of $u$ is thus contradicted also
in this case.
\end{proof}

\paragraph{Acknowledgment}
We wish to thank the referee for their careful reading of the paper, and for their
valuable comments.

\section*{Compliance with Ethical Standards}

\subsection*{Funding}

This research was partly funded by:

\begin{enumerate}
\item Research Project 201758MTR2 of the Italian Ministry of University and
Research (MIUR) Prin 2017 ``Direct and inverse problems for partial differential equations: theoretical aspects and applications'';
\item GNAMPA of the Italian INdAM -- National Institute of High Mathematics
(grant number not available);
\item Grant P201-18-00580S of the Czech Science Foundation.
\end{enumerate}

\subsection*{Conflict of Interest}

The authors declare that they have no conflict of interest.

\begin{small}

\end{small}

\end{document}